\documentclass[12pt]{amsart}
\usepackage{amssymb}
\usepackage{amsmath}
\usepackage{bbm} 
\usepackage[dvips]{graphicx}
\usepackage{hyperref}
\usepackage[capitalize,nameinlink,noabbrev,nosort]{cleveref}
\usepackage[dvipsnames]{xcolor}
\usepackage{todonotes}
\usepackage{mathtools} 
\usepackage{enumitem}
\usepackage[headings]{fullpage}
\usepackage{ytableau}
\usepackage{yhmath}
\usepackage{mathdots}

\usepackage{tikz}
\usetikzlibrary{arrows}
\usetikzlibrary{decorations.markings}
\usetikzlibrary{tikzmark,decorations.pathreplacing}
\usetikzlibrary{shapes.geometric}
\numberwithin{equation}{section}
\usepackage{float}
\makeatletter
\def\Ddots{\mathinner{\mkern1mu\raise\p@
\vbox{\kern7\p@\hbox{.}}\mkern2mu
\raise4\p@\hbox{.}\mkern2mu\raise7\p@\hbox{.}\mkern1mu}}
\makeatother

\newcommand{\ds}{\displaystyle}
\DeclareMathOperator{\Sp}{Sp}

\DeclareMathOperator{\GL}{GL}
\DeclareMathOperator{\OO}{SO}

\renewcommand{\sp}{\mathrm{sp}}
\DeclareMathOperator{\oo}{so}
\renewcommand{\OE}{\mathrm{O}}
\renewcommand{\oe}{\mathrm{o}^{\text{even}}}

\DeclareMathOperator{\sgn}{sgn}
\DeclareMathOperator{\inv}{inv}
\newcommand{\floor}[1]{\left\lfloor #1 \right\rfloor}

\newcommand{\core}[2]{\mathrm{core}_{#2}{(#1)}}
\newcommand{\quo}[2]{\mathrm{quo}_{#2}{(#1)}}
\newcommand{\x}{\bar{x}}
\newcommand{\X}{\overline{X}}

\DeclareMathOperator{\rk}{rk}
\DeclareMathOperator{\rev}{rev}

\newtheorem{thm}{Theorem}[section]
\newtheorem{prop}[thm]{Proposition}
\newtheorem{cor}[thm]{Corollary}
\newtheorem{lem}[thm]{Lemma}
\newtheorem{defn}[thm]{Definition}
\newtheorem{rem}[thm]{Remark}

\newtheorem{eg}[thm]{Example}

\crefname{thm}{Theorem}{Theorems}

\title
{Factorization of classical characters twisted by roots of unity}
\author{Arvind Ayyer}
\address{Arvind Ayyer, Department of Mathematics, 
Indian Institute of Science, Bangalore  560012, India.}
\email{arvind@iisc.ac.in}

\author{Nishu Kumari}
\address{Nishu Kumari, Department of Mathematics, 
Indian Institute of Science, Bangalore  560012, India.}
\email{nishukumari@iisc.ac.in}

\date{\today}

\begin{document}

\begin{abstract}
For a fixed integer $t \geq 2$, we consider the irreducible characters of representations of the classical groups of types A, B, C and D, namely $\GL_{tn}, \OO_{2tn+1}, \Sp_{2tn}$ and $\OE_{2tn}$, 
evaluated at elements $\omega^k x_i$ for $0 \leq k \leq t-1$ and $1 \leq i \leq n$, where $\omega$ is a primitive $t$'th root of unity.
The case of $\GL_{tn}$ was considered by D. J. Littlewood (AMS press, 1950) and independently by D. Prasad (Israel J. Math., 2016).
In this article, we give a uniform approach for all cases.
We also look at $\GL_{tn+1}$ where we specialize the elements as before and set the last variable to $1$.
In each case, we characterize partitions for which the character value is nonzero in terms of what we call $z$-asymmetric partitions, where $z$ is an integer which depends on the group.
Moreover, if the character value is nonzero, we prove that it factorizes into characters of smaller classical groups.
The proof uses Cauchy-type determinant formulas for these characters and involves a careful study of the beta sets of partitions.
We also give product formulas for general $z$-asymmetric partitions and $z$-asymmetric $t$-cores.
Lastly, we show that there are infinitely many $z$-asymmetric $t$-cores for $t \geq z+2$.
\end{abstract}


\maketitle 

\section{Introduction} 

The characters of irreducible representations of the classical families of groups, namely the general linear, symplectic and orthogonal groups are amazing families of symmetric Laurent polynomials indexed by integer partitions. In particular, the {polynomial characters} of the general linear groups are the \emph{Schur polynomials}, which are extremely well-studied. They form the most natural basis of the ring of symmetric functions, which are orthonormal with respect to the standard \emph{Hall inner product}. For background, see \cite{macdonald-2015}.

These families of Laurent polynomials also satisfy nontrivial relations, which are not well-understood from the point of view of representation theory. For instance, it was shown in \cite{CiuKra09} that the
Schur polynomial for a rectangular partition in $2n$ variables specialized to the last $n$ variables being reciprocals of the first $n$ variables becomes a product of two other classical characters. In some cases, this is the product of a symplectic and an even orthogonal character, and in the other, {it is} the product of two odd orthogonal characters. 
Similar factorization results were obtained in \cite{AyyBehFis16} for so-called \emph{double staircase partitions}, i.e. partitions of the form $(k,k,k-1,k-1,\dots,1,1)$ or $(k,k-1,k-1,\dots,1,1)$.
This kind of factorization was generalized in \cite{ayyer-behrend-2019} 
for a large class of partitions and further, to skew-Schur functions, i.e. induced characters, in \cite{ayyer-fischer-2020}. 

In a different direction, Littlewood~\cite{littlewood-1950} and independently Prasad~\cite{prasad-2016} considered factorizations of Schur polynomials in $2n$ variables where the last $n$ variables were negatives of the first $n$ variables motivated by a celebrated result of Kostant~\cite{kostant-1976}. {They} showed that such a factorization is nonzero if and only if the corresponding $2$-core is empty, and if it is nonzero, it factors into characters for the $2$-quotients; see \cref{sec:summary-results} for the definitions. {They} further generalized this result to $tn$ variables, for $t \geq 2$ a fixed positive integer, specialized to $(\exp(2 \pi \iota k/t) x_j)_{0 \leq k \leq t-1, 1 \leq j \leq n}$, obtaining similar results. 
We will think of these as twisted characters, where the twists are by all the $t$'th roots of unity.

We note in passing that Schur polynomials evaluated at roots of unity and their powers have been considered in \cite{macdonald-2015,rhoades-2010}.

In this work, we generalize Littlewood's results to other classical groups. We first consider characters of $\GL_{tn+1}$, where we add an extra variable set to $1$. This is stated as \cref{thm:schur-1}.
We then consider the classical groups $\Sp_{2tn}, \OE_{2tn}$ and $\OO_{2tn+1}$ and obtain factorizations for their characters under the same specialization as that of Littlewood. These are stated as \cref{thm:sympfact}, \cref{thm:eorthfact} and \cref{thm:oorthfact} respectively.
Our proofs are more involved for the following reason. For the general linear group, there is only one possible value of the $t$-core for which the twisted character is nonzero, namely the empty partition. For the other classical characters, there are many possible values of the $t$-core for which the character is nonzero. We will show that these are $t$-cores which can be written in Frobenius coordinates as $(\alpha | \alpha + z)$, where the value of $z$ depends on the group, and which we call $z$-asymmetric partitions. For the study here, $z \in \{-1,0,1\}$.

The plan for the rest of the paper is as follows.
We give all the definitions, statements of our results and illustrative examples in \cref{sec:summary-results}.
Our strategy is to give uniform proofs of these results. To that end, we formulate results on beta sets, generating functions and determinant identities in \cref{sec:background}.
We prove the Schur factorizations in \cref{sec:schur} including a self-contained proof of Littlewood's result, \cref{thm:schur-fac}.
We prove the new factorizations of classical characters in \cref{sec:fact other}. We give all the details for the symplectic factorization, and are much more sketchy about the even and odd orthogonal character factorization.
Finally, we prove generating function formulas for $z$-asymmetric partitions
and $z$-asymmetric $t$-cores in \cref{sec:gf}. In particular, we will prove \cref{thm:inf-cores} there, showing that there are infinitely many $t$-cores at which our character values are nonzero for $t \geq 3$.

\section{Summary of results}
\label{sec:summary-results}

Throughout, we will fix $t$ to be an integer greater than or equal to $2$. 
Let $\omega$ be a primitive $t$'th root of unity, i.e. $\omega^t = 1$. 
We will also use $n$ for a fixed positive integer and let $X = (x_1,\dots,x_n)$ be a tuple of commuting indeterminates. For any integer $j$,
we set $X^j = (x_1^j,\dots,x_n^j)$, and for $a \in \mathbb{R}$, set $aX = (a x_1,\dots,a x_n)$.
Define $\x = 1/x$ for an indeterminate $x$ and write $\X = (\x_1,\dots,\x_n)$. 

Recall that a {\em partition} $\lambda$ is a weakly decreasing sequence of nonnegative integers
$\lambda = (\lambda_1,\dots,\lambda_m)$.
The {\em length} of a partition $\lambda$, denoted $\ell(\lambda)$ is the number of nonzero parts of $\lambda$. By $a + \lambda$, we will mean the partition $(a + \lambda_1,\dots, a + \lambda_m)$. For a partition $\lambda$ and an integer $m$ such that $\ell(\lambda) \leq m$, 
define the {\em beta-set of $\lambda$} by $\beta(\lambda) \equiv \beta(\lambda,m) = (\beta_1(\lambda,m),\dots,\beta_m(\lambda,m))$ by $\beta_i(\lambda) = \lambda_i+m-i$.
We will use the convention that we will write $\beta(\lambda)$ whenever $m$ is clear from the context.

We write down the explicit Weyl character formulas for the infinite families of classical groups. 
For completeness, we also write down the Weyl denominator identities in each case.
See \cite{FulHar91} for more details and background.
For a partition $\lambda=(\lambda_1,\ldots,\lambda_n)$, the \emph{Schur polynomial} or \emph{general linear (type A) character} of $\GL_n$ is given by
\begin{equation}
\label{gldef}
s_\lambda(X)=\frac{\displaystyle\det_{1\le i,j\le n}\Bigl(x_i^{\beta_j(\lambda)}\Bigr)}
{\displaystyle\det_{1\le i,j\le n}\Bigl(x_i^{n - j}\Bigr)}.
\end{equation}
The denominator is the standard Vandermonde determinant,
\begin{equation}
\label{gldenom}
\det_{1\le i,j\le n}\Bigl(x_i^{n - j}\Bigr)
= \prod_{1\le i<j\le n}(x_i-x_j).
\end{equation}
The \emph{odd orthogonal (type B) character} of the group $\OO(2n+1)$ at the representation indexed by $\lambda = (\lambda_1,\dots,\lambda_n)$ is given by
\begin{equation}
\label{oodef}
\oo_\lambda(X)=
\frac{\ds \det_{1\le i,j\le n}\Bigl(x_i^{\beta_j(\lambda)+1/2}-\x_i^{\beta_j(\lambda)+1/2}\Bigr)}
{\ds \det_{1\le i,j\le n}\Bigl(x_i^{n-j+1/2}-\x_i^{n-j+1/2}\Bigr)} 
= \frac{\ds \det_{1\le i,j\le n}\Bigl(x_i^{\beta_j(\lambda)+1}-\x_i^{\beta_j(\lambda)}\Bigr)}
{\ds \det_{1\le i,j\le n}\Bigl(x_i^{n-j+1}-\x_i^{n-j}\Bigr)},
\end{equation}
and the denominator here is
\begin{equation}
\label{oodenom}
\det_{1\le i,j\le n}\Bigl(x_i^{n-j+1/2}-\x_i^{n-j+1/2}\Bigr)
= \prod_{i=1}^n\bigl(x_i^{1/2}-\x_i^{1/2}\bigr)\,\prod_{1\le i<j\le n}(x_i+\x_i-x_j-\x_j).
\end{equation}
The \emph{symplectic (type C) character} of the group $\Sp(2n)$ at the representation indexed by $\lambda = (\lambda_1,\dots,\lambda_n)$ is given by
\begin{equation}
\label{spdef}
\sp_\lambda(X)=
\frac{\ds \det_{1\le i,j\le n}\Bigl(x_i^{\beta_j(\lambda)+1}-\x_i^{\beta_j(\lambda)+1}\Bigr)}
{\ds \det_{1\le i,j\le n}\Bigl(x_i^{n-j+1}-\x_i^{n-j+1}\Bigr)},
\end{equation}
and the denominator here is
\begin{equation}
\label{spdenom}
\det_{1\le i,j\le n}\Bigl(x_i^{n-j+1}-\x_i^{n-j+1}\Bigr)
= \prod_{i=1}^n(x_i-\x_i)\,\prod_{1\le i<j\le n}(x_i+\x_i-x_j-\x_j).
\end{equation}
Lastly, the \emph{even orthogonal (type D) character} of the group $\OE(2n)$ at the representation indexed by $\lambda = (\lambda_1,\dots,\lambda_n)$ is given by
\begin{equation}
\label{oedef}
\oe_\lambda(X)=
\frac{{2} \displaystyle\det_{1\le i,j\le n}\Bigl(x_i^{\beta_j(\lambda)}+\x_i^{\beta_j(\lambda)}\Bigr)}
{\displaystyle (1+\delta_{\lambda_n,0})
\det_{1\le i,j\le n}\Bigl(x_i^{n-j}+\x_i^{n-j}\Bigr)},
\end{equation}
where $\delta$ is the Kronecker delta. 
{
The extra factor in the denominator arises because of the difference in the representation theory of $\OE(2n)$ and $\OO(2n)$; see~\cite[p. 411]{FulHar91} and \cite[pp. 311--312]{proctor-1994} for the precise details.
} 
The determinant here factorizes as
\begin{equation}
\label{oedenom}
\det_{1\le i,j\le n}\Bigl(x_i^{n-j}+\x_i^{n-j}\Bigr)
= {2} \prod_{1\le i<j\le n}(x_i+\x_i-x_j-\x_j).
\end{equation}

Notice that 
\[
s_\lambda(x_1,\ldots,x_n) = \sp_\lambda(x_1,\allowbreak\dots,x_n)  = \oo_\lambda(x_1,\dots,x_n) = \oe_\lambda(x_1,\dots,x_n) = 0, \quad
\text{if } n < \ell(\lambda).
\]
There is a general relation between even and odd orthogonal characters as follows. Suppose $(\lambda_1,\dots,\lambda_n)$ is an integer partition. Then
\begin{equation}
 \label{oeshifted}
\oe_{(\lambda_1+1/2,\ldots,\lambda_n+1/2)}(x_1,\ldots,x_n)=
(-1)^{\sum_{i=1}^n\lambda_i}
\prod_{i=1}^n\bigl(x_i^{1/2}+\x_i^{1/2}\bigr)
\:\oo_\lambda(-x_1,\ldots,-x_n).   
\end{equation}

A partition $\lambda$ can be represented pictorially as a \emph{Young diagram}, whose $i$'th row contains $\lambda_i$ left-justified boxes. We will use the so-called English notation where the first row is on top. For example, the Young diagram of $(4,2,2,1)$ is
\begin{equation}
\label{eg-ydiag}
\ydiagram{4,2,2,1}.
\end{equation}
For a partition, $\lambda$, the {\em conjugate partition}, denoted $\lambda'$, is the partition whose Young diagram is obtained by transposing the Young diagram of $\lambda$.
A \emph{border strip} is a connected subdiagram of the Young diagram of $\lambda$ which contains no $2 \times 2$ block of squares. Therefore, successive rows and columns of a border strip overlap by exactly one box. For example,
\[
\ydiagram{2 + 2, 1 + 2, 1 + 1, 1 + 1}
\]
is a border strip of length 6.

\begin{defn}
\label{def:tcore}
The \emph{$t$-core} of the partition $\lambda$, denoted $\core \lambda t$, is the partition obtained by successively removing border strips of size $t$ from the Young diagram of $\lambda$. 
\end{defn}

{In the example given in \eqref{eg-ydiag}, we see that $\core {(4,2,2,1)} 2 = (2,1)$ after three border strip removals.} It is a nontrivial fact that the resulting partition is independent of the order of removal; see \cite[Theorem~11.16]{loehr-2011}. For example, the only $2$-cores are staircase shapes, i.e. of the form $(k,k-1,\dots,1,0)$, $k \in \mathbb{N}$.

For a cell $c = (i,j)$ in (the Young diagram of) $\lambda$, the \emph{hook length} is given by $h_c = \lambda_i - i + \lambda'_j - j + 1$, which is the {total number of cells in its row to the right and those in its column below it including the cell itself.} 
The \emph{content} of $c$ is $j-i$. The \emph{arm} (resp. \emph{leg}) of $c$ is the rightmost (resp. bottommost) cell in its row (resp. column). For example, the hook lengths and contents of the running example are
\begin{equation}
\label{eg-hook-content}
\begin{ytableau}
7 & 5 & 2 & 1 \\
4 & 2 \\
3 & 1 \\
1
\end{ytableau}
\quad
\text{and}
\quad
\begin{ytableau}
0 & 1 & 2 & 3 \\
-1 & 0 \\
-2 & -1 \\
-3
\end{ytableau}.
\end{equation}

\begin{defn}
\label{def:tquo}
The \emph{$t$-quotient} of $\lambda$ is a $t$-tuple of partitions denoted 
$\quo \lambda t = (\lambda^{(0)},\dots, \allowbreak \lambda^{(t-1)})$ obtained using the Young diagram of $\lambda$. The $i$'th element of this tuple is obtained as follows. Take all cells $c$ whose hook length is divisible by $t$ and whose arm has content congruent to $i \pmod t$. It is a nontrivial fact that this collection of cells forms a Young subdiagram of $\lambda$. The corresponding partition is $\lambda^{(i)}$.
\end{defn}

From \eqref{eg-hook-content}, we see that $\quo {(4,2,2,1)} 2 = ((2),(1))$.
Macdonald defines the $t$-core and $t$-quotient alternately using the beta-set and we recall this construction. 
Let $\lambda$ be a partition with $\ell(\lambda) \leq m$.
For $0 \leq i \leq t-1$,
let $n_{i}(\lambda) \equiv n_{i}(\lambda,m)$ be the number of parts of $\beta(\lambda)$ congruent to $i \pmod{t}$ and  $\beta_j^{(i)}(\lambda)$, $1 \leq j \leq n_{i}(\lambda)$ be the $n_{i}(\lambda)$ parts of $\beta(\lambda)$ congruent to $i \pmod t$ in decreasing order.

\begin{prop}[{\cite[Example~I.1.8]{macdonald-2015}}]
\label{prop:mcd-t-core-quo}
Let $\lambda$ be a partition with $\ell(\lambda) \leq m$.

\begin{enumerate}
\item The $m$ numbers $tj+i$, where $0 \leq j \leq n_{i}(\lambda)$ and $0 \leq i \leq t-1$, are all distinct. Arrange them in descending order, say $\tilde{\beta}_1>\dots>\tilde{\beta}_m$. Then the $t$-core of $\lambda$ has parts $(\core \lambda t)_i=\tilde{\beta}_i-m+i$.  
Thus, $\lambda$ is a $t$-core if and only if these $m$ numbers $tj+i$, where $0 \leq j \leq n_{i}(\lambda)$ and $0 \leq i \leq t-1$ form its beta-set $\beta(\lambda)$.

\item The parts $\beta_j^{(i)}(\lambda)$ may be written in the form $t\tilde{\beta}_j^{(i)}+i$, $1 \leq j \leq n_{i}(\lambda)$,
where $\tilde{\beta}_1^{(i)}>\dots>\tilde{\beta}_{n_{i}(\lambda)}^{(i)}\geq 0$. 
Let $\lambda_j^{(i)}=\tilde{\beta}_j^{(i)}-n_{i}(\lambda)+j$, 
so that $\lambda^{(i)}=(\lambda_1^{(i)},\dots,\lambda_{n_{i}(\lambda)}^{(i)})$ is a partition. 
Then the $t$-quotient $\quo \lambda t$ of $\lambda$ is a cyclic permutation of $\lambda^{\star} = (\lambda^{(0)},\lambda^{(1)},\dots,\lambda^{(t-1)})$.
The effect of changing $m\geq \ell(\lambda)$ is to permute the $\lambda^{(j)}$ cyclically, so that $\lambda^{\star}$ should perhaps be thought of as a `necklace' of partitions.
\end{enumerate}
\end{prop}

\begin{rem}
\label{rem:necklace}
We note that Macdonald's definition of the $t$-quotient is not identical to that of \cref{def:tquo}, but is equal up to a cyclic shift. In particular, if
$\quo \lambda t = (\lambda^{(0)},\dots,\lambda^{(t-1)})$ and $m$ increases by $1$ in \cref{prop:mcd-t-core-quo}, the new $t$-quotient will be $(\lambda^{(t-1)},\lambda^{(0)},\dots,\lambda^{(t-2)})$. 
\end{rem}

For a partition of length at most $tn$, 
let $\sigma_{\lambda} \in S_{tn}$ be the permutation that rearranges the parts of $\beta(\lambda)$ such that
\begin{equation}
\label{sigma-perm}
 \beta_{\sigma_{\lambda}(j)}(\lambda) \equiv q \pmod t, \quad  \sum_{i=0}^{q-1} n_{i}(\lambda)+1 \leq j \leq \sum_{i=0}^{q} n_{i}(\lambda),   
\end{equation}
arranged in decreasing order for each $q \in \{0,1,\dots,t-1\}$.
For the empty partition,  $\beta(\emptyset,tn)=(tn-1,tn-2,\dots,0)$ with $n_{q}(\emptyset,tn)=n$, $0 \leq q \leq t-1$ and 

\begin{equation}
\label{sigma0}
\sigma_{\emptyset}={\displaystyle (t, \dots , nt, t-1, \dots ,nt-1, \dots,1, \dots ,(n-1)t+1}),
\end{equation}
in one line notation with $\sgn(\sigma_{\emptyset})=(-1)^{\frac{t(t-1)}{2}\frac{n(n+1)}{2}}$.

The {\em (Frobenius) rank} of a partition $\lambda$, denoted $\rk(\lambda)$, is the largest integer $k$ such that $\lambda_k \geq k$. 
The {\em Frobenius coordinates} of $\lambda$ are a pair of distinct partitions, denoted $(\alpha | \beta)$, of length $\rk(\lambda)$ given by $\alpha_i = \lambda_i - i$ and
$\beta_j = \lambda'_j - j$. 
For example, the Frobenius coordinates of our running example {$(4,2,2,1)$} in \eqref{eg-ydiag} are $(3,0|3,1)$.

We will consider classical characters evaluated at elements twisted by all the $t$'th roots of unity. 
The first result in this direction is due to D. Littlewood and independently D. Prasad for $\GL_{tn}$.
We will denote our indeterminates by $X,\omega X,\omega^2X, \dots ,\omega^{t-1}X$, where we recall that $X = (x_1,\dots,x_n)$ and $\omega$ is a primitive $t$'th root of unity. 

\begin{thm}[{\cite[Equation (7.3;3)]{littlewood-1950}, \cite[Theorem 2]{prasad-2016}}]
\label{thm:schur-fac}
Let $\lambda$ be a partition of length at most $tn$ indexing an irreducible representation of $\GL_{tn}$ and $\quo \lambda t  = (\lambda^{(0)},\dots,\lambda^{(t-1)})$.
Then the $\GL_{tn}$-character $s_{\lambda}(X,\omega X, \dots ,\omega^{t-1}X)$ is as follows.
\begin{enumerate}
    \item  If $\core \lambda t$ is not {empty}, then
\begin{equation}
    \label{th:1.1}
    s_{\lambda}(X,\omega X, \dots ,\omega^{t-1}X) = 0.
\end{equation}
\item If $\core \lambda t$ is {empty}, then 
\begin{equation}
  \label{th:1.2} s_{\lambda}(X,\omega X, \dots ,\omega^{t-1}X) \\ =  (-1)^{\frac{t(t-1)}{2}\frac{n(n+1)}{2}}\sgn(\sigma_{\lambda})\prod_{i=0}^{t-1} s_{\lambda^{(i)}}(X^t).
  \end{equation}
\end{enumerate}
\end{thm}

In other words, the nonzero $\GL_{tn}$-character is a product of $t$ $\GL_{n}$ characters. We will give a self-contained proof of this result in \cref{sec:schur}. 
We note that \cref{thm:schur-fac} for $X=(1)$ is due to Macdonald~\cite[Chapter I.3, Example  17(a)]{macdonald-2015}, where the Schur polynomial on the right hand side of \eqref{th:1.2} is 1 for each $i \in [0,t-1]$.

\begin{eg} For $t=2$, \cref{thm:schur-fac} says that the character of the group $\GL_{2}$ (i.e., $n=1$) of the representation indexed by the partition $(a,b)$, $a \geq b \geq 0$, evaluated at $(x,-x)$ is nonzero if and only if $a$ and $b$ have the same parity. If $a$ and $b$ are both odd, then
\[
s_{(a,b)}(x,-x)=-s_{(\frac{a+1}{2})}(x^2) s_{(\frac{b-1}{2})}(x^2),
\]
and if $a$ and $b$ are both even, then
\[
s_{(a,b)}(x,-x)=s_{(\frac{b}{2})}(x^2) s_{(\frac{a}{2})}(x^2).
\]
\end{eg}

{We now consider Schur factorizations where there is an extra variable set equal to $1$. We were informed later that this result was obtained earlier by Littlewood using very different notation.}
To state our result, we note that the only $t$-cores of length at most one are the partitions $(c)$ for $0 \leq c \leq t-1$. 

\begin{thm}[{Littlewood~\cite[Chapter VII, Section IX]{littlewood-1950}}]
\label{thm:schur-1}
Let $\lambda$ be a partition of length at most $tn+1$ indexing an irreducible representation of $\GL_{tn+1}$ and $\quo \lambda t  = (\lambda^{(0)},\dots,\lambda^{(t-1)})$.
Then the $\GL_{tn+1}$-character $s_{\lambda}(X,\omega X, \dots, \allowbreak \omega^{t-1}X,x)$ is as follows.

\begin{enumerate}

\item  If $\ell(\core \lambda t) > 1$, then
\begin{equation}
    \label{th:7.1}
    s_{\lambda}(X,\omega X, \dots ,\omega^{t-1}X,x) = 0.
\end{equation}

\item If $\core \lambda t = (c)$, for some $0 \leq c \leq t-1$, then 
\begin{equation}
  \label{th:7.2} s_{\lambda}(X,\omega X, \dots ,\omega^{t-1}X,x) \\ = 
  (-1)^{\frac{t(t-1)}{2}\frac{n(n+1)}{2}-cn}
  \sgn(\sigma^c_{\lambda}) \, x^c \, s_{\lambda^{(c)}}(X^t,x^t) \prod_{\substack{i=0\\i \neq c}}^{t-1} s_{\lambda^{(i)}}(X^t).
  \end{equation}
\end{enumerate}
\end{thm}

{This result for $t =2$, formulated slightly differently, is present in Littlewood~\cite[Chapter VII, Section VIII]{littlewood-1950} for $x=1$ and in ~\cite[Theorem 4.5]{LuPra21}.}

\begin{eg} 
For $t=2$, \cref{thm:schur-1} says that the character of the group $\GL_{3}$ (i.e., $n=1$) of the representation indexed by the partition $(a,b,c)$, $a \geq b \geq c \geq 0$, evaluated at $(x_1,-x_1,x)$ is non-zero if and only if $a$ and $b$ have the same parity or $a$ and $c$ have the opposite parity. If $\core {a,b,c} 2$ is empty, then
\[
s_{(a,b,c)}(x_1,-x_1,x)=
\begin{cases}
-s_{(\frac{a}{2},\frac{b+1}{2})}(x_1^2,x^2) s_{(\frac{c-1}{2})}(x_1^2) & \text{$a$  even, $b$ and $c$ odd,}\\
-s_{(\frac{a}{2},\frac{c}{2})}(x_1^2,x^2) s_{(\frac{b}{2})}(x_1^2) & a,b,c \text{ even,}\\
s_{(\frac{b-1}{2},\frac{c}{2})}(x_1^2,x^2) s_{(\frac{a+1}{2})}(x_1^2) & \text{$a$ and $b$ odd, $c$ even,}
\end{cases}
\]
and if $\core {a,b,c} 2 = (1)$, then
\[
s_{(a,b,c)}(x_1,-x_1,x)=
\begin{cases}
x \, s_{(\frac{a-1}{2},\frac{b}{2})}(x_1^2,x^2) s_{(\frac{c}{2})}(x_1^2) & \text{$a$ odd, $b$ and $c$ even,}\\
x \, s_{(\frac{a-1}{2},\frac{c-1}{2})}(x_1^2,x^2) s_{(\frac{b+1}{2})}(x_1^2) & a,b,c \text{ odd,}\\
-x\,  s_{(\frac{b-2}{2},\frac{c-1}{2})}(x_1^2,x^2) s_{(\frac{a+2}{2})}(x_1^2) & \text{ $a$ and $b$ even, $c$ odd.}
\end{cases}
\]
\end{eg}

We now generalize \cref{thm:schur-fac} to other classical characters. We first need some definitions.

\begin{defn}
\label{def:z-asym}
Let $z$ be a nonnegative integer. We say that a partition $\lambda$ is \emph{$z$-asymmetric} if $\lambda=(\alpha|\, \alpha+z)$, in Frobenius coordinates for some strict partition $\alpha$. More precisely, 
$\lambda = (\alpha|\beta)$ where $\beta_i = \alpha_i + z$ for $1 \leq i \leq \rk(\lambda)$.
\end{defn}

\begin{defn}
\label{def:symp-core}
A $1$-asymmetric partition is said to be {\em symplectic}\footnote{While this terminology also seems to be have been used for partitions whose odd parts have even multiplicity~\cite{jiang-liu-2015}, it does not seem widespread.}.
In addition, if a symplectic partition is also a $t$-core, we call it a {\em symplectic $t$-core}.
\end{defn}

Note that the empty partition is vacuously symplectic. For example, the only symplectic partitions of $6$ are $(3,1,1,1)$ and $(2,2,2)$, and the first few symplectic $3$-cores are $(1,1)$, $(2,1,1), (4,2,2,1,1)$ and $(5,3,2,2, \allowbreak 1,1)$.  

For the symplectic case, we take $G = \Sp_{2tn}$, the symplectic group of $(2tn) \times (2tn)$ matrices.  
To state our results, it will be convenient to define, for $\lambda = (\lambda_1,\dots,\lambda_k)$, the reverse of $\lambda$ as $\rev(\lambda) = (\lambda_k,\dots,\lambda_1)$. Further, if $\mu = (\mu_1,\dots,\mu_j)$ is another partition such that $\mu_1 \leq \lambda_k$, then we write the concatenated partition $(\lambda,\mu) = (\lambda_1,\dots,\lambda_k,\mu_1,\dots,\mu_j)$.

\begin{thm}
\label{thm:sympfact} 
Let $\lambda$ be a partition of length at most $tn$ indexing an irreducible representation of $\Sp_{2tn}$ and $\quo \lambda t  = (\lambda^{(0)},\dots,\lambda^{(t-1)})$.
Then the $\Sp_{2tn}$-character $\sp_{\lambda}(X,\omega X, \allowbreak \dots ,\omega^{t-1}X)$ is given as follows.

\begin{enumerate}

\item  If $\core \lambda t$ is not a symplectic $t$-core, then
\begin{equation}
\label{symp-fact-1}
    \sp_{\lambda}(X,\omega X, \dots ,\omega^{t-1}X) = 0.
\end{equation}

\item If $\core \lambda t$ is a symplectic $t$-core with rank r, then 
\begin{equation}
    \begin{split}
      \label{symp-fact-2}
\sp_{\lambda}(X,\omega X, \dots ,\omega^{t-1}X) =  (-1)^{\epsilon} \sgn (\sigma_{\lambda})\, \sp_{\lambda^{(t-1)}}(X^t) 
 & \prod_{i=0}^{\floor{\frac{t-3}{2}}} s_{\mu^{(1)}_i}(X^t,{\X}^t) \\ & \times 
\begin{cases}
\oo_{\lambda^{\left(\frac{t}{2}-1\right)}}(X^t) & t \text{ even} ,\\
1 & t \text{ odd},
\end{cases}  
    \end{split}
\end{equation}
where 
\[
\epsilon= \ds -\sum_{i=\floor{\frac{t}{2}}}^{t-2} \binom{n_{i}(\lambda)+1}2 + \begin{cases}
\frac{n(n+1)}{2}+nr & t \text{ even},
\\0  & t \text{ odd},
\end{cases}
\]
and $\ds \mu^{(1)}_i=  \lambda^{(t-2-i)}_1 +
\left(\lambda^{(i)}, 0,\dots,0, -\rev(\lambda^{(t-2-i)})\right)$ has $2n$ parts for 
$0 \leq i \leq \floor{\frac{t-3}{2}}$.
\end{enumerate}
\end{thm}

Again, nonzero $\Sp_{2tn}$ characters are a product of characters, but this time there are $\floor{(t-1)/2}$ $\GL_{2n}$ characters, one $\Sp_{2n}$ character and, if $t$ is even, one additional $\OO_{2n+1}$ character.
As mentioned above, the only $2$-cores are self-conjugate. Therefore, this character when $t=2$ is nonzero if and only if $\core \lambda 2 = \emptyset$.

\begin{eg}
For $t = 2$, \cref{thm:sympfact} says that the character of the group $\Sp(4)$ $(n=1)$ of the representation indexed by the partition $(a,b)$, $a \geq b \geq 0$, evaluated at  $(x, -x)$ is nonzero if and only if $a$ and $b$ have the same parity. 
If $a$ and $b$ are both odd, then 
\[
\sp_{(a,b)}(x, -x) = - \sp_{(\frac{b-1}{2})}(x^2) \oo_{(\frac{a+1}{2})}(x^2),
\]
and if $a$ and $b$ are both even, then 
\[
\sp_{(a,b)}(x, -x) = \sp_{(\frac{a}{2})}(x^2) \oo_{(\frac{b}{2})}(x^2).
\]
Notice that all the characters on the right-hand side are for the groups $\Sp(2)$ and $\OO(3)$, and in both cases, the partitions indexing them are the $2$-quotients and of length 1.
\end{eg}

We also give a concrete example.

\begin{eg}
\label{eg:t=3}
Let $n=2, t=3$ and consider the partition $\lambda = (3,2,1,1,1)$ so that $\beta (\lambda) = (8,6,4,3,2,0)$.
Hence, $n_{0}(\lambda,6) = 3, n_{1}(\lambda,6) = 1$ and $n_{2}(\lambda,6) = 2$.
Hence, it has $3$-core equal to $(1,1)$, and its symplectic character is nonzero.
With $X = (x_1,x_2)$, $\sp_{\lambda}(X,\omega_3 X,\omega_3^2 X)$ is given by
\[
\left(\frac{(x_1+x_2) (x_1 x_2+1) \left(x_1^2-x_2 x_1+x_2^2\right) \left(x_1^2 x_2^2-x_1 x_2+1\right)}{x_1^3 x_2^3} \right)^2.
\]
Since $\quo \lambda 3 = (\emptyset, (1), (1))$, $\mu^{(1)}_0 = 0 + (1,0,0,0) = (1)$ and
we need to calculate $\sp_{(1)}(X^3)$ and $s_{(1)}(X^3, \X^3)$. {These are the characters of $\Sp(4)$, $\OO(5)$
respectively, corresponding to the partition $(1, 0)$}. It turns out that both are equal to
\[
\frac{(x_1+x_2) (x_1 x_2+1) \left(x_1^2-x_2 x_1+x_2^2\right) \left(x_1^2 x_2^2-x_1 x_2+1\right)}{x_1^3 x_2^3},
\]
verifying \cref{thm:sympfact}.
\end{eg}

\begin{defn}
\label{def:orth-core}
A $(-1)$-asymmetric partition is said to be {\em orthogonal}.
In addition, if an orthogonal partition is also a $t$-core, we call it an {\em orthogonal $t$-core}.
\end{defn}

Our notion of an orthogonal partition is {the} same as Macdonald's \emph{double of $\alpha$}~\cite[p. 14]{macdonald-2015}, and Garvan--Kim--Stanton's \emph{doubled partition of $\alpha$}, denoted $\alpha\alpha$~\cite[Sec. 8]{garvan-kim-stanton-1990}. The first few orthogonal $3$-cores are $(2)$, $(3,1), (5,3,1,1)$ and $(6,4,2,\allowbreak 1,1)$, which are precisely the conjugates of the symplectic $3$-cores listed earlier. 
Then our result for factorization of even orthogonal characters is as follows.

For the even orthogonal case, we take $G = \OE_{2tn}$, the orthogonal group of $(2tn) \times (2tn)$ square matrices.

\begin{thm}
\label{thm:eorthfact}
Let $\lambda$ be a partition of length at most $tn$ indexing an irreducible representation of $\OE_{2tn}$ and $\quo \lambda t  = (\lambda^{(0)},\dots,\lambda^{(t-1)})$. 
Then the $\OE_{2tn}$ character $\oe_{\lambda}(X,\omega X,\allowbreak \dots ,\omega^{t-1}X)$ 
is given as follows. 

\begin{enumerate}
    \item  If $\core \lambda t$ is not an orthogonal $t$-core, then
\begin{equation}
    \label{th:5.1}
    \oe_{\lambda}(X,\omega X, \dots ,\omega^{t-1}X) = 0.
\end{equation}
\item If $\core \lambda t$ is an orthogonal $t$-core with rank r, then
\begin{equation}
    \begin{split}
     \label{th:5.2}  \oe_{\lambda}(X,\omega X, \dots ,\omega^{t-1}X)  =  (-1)^{\epsilon}\sgn
(\sigma_{\lambda}) \,
 & \oe_{\lambda^{(0)}}(X^t)   \prod_{i=1}^{\floor{\frac{t-1}{2}}} s_{\mu^{(2)}_i}(X^t,{\X}^t) \\&  \times 
\begin{cases}
(-1)^{\sum_{i=1}^n \lambda_i^{(t/2)}}\displaystyle {\oo_{\lambda^{(t/2)}}(-X^t)} & t \text{ even},\\
1 & t \text{ odd},
\end{cases}   
    \end{split}
\end{equation}
where 
\[
\epsilon=-\sum_{i=\floor{\frac{t+2}{2}}}^{t-1} \binom{n_{i}(\lambda)}{2} + \begin{cases}
\frac{n(n+t-1)}{2}+nr & t \text{ even}, \\ 
\frac{(t-1)n}{2} & t \text{ odd},
\end{cases}
\]
and $\ds \mu^{(2)}_i=  \lambda^{(t-i)}_1 + \left(\lambda^{(i)}, 0,\dots,0, -\rev(\lambda^{(t-i)})\right)$ has $2n$ parts for 
$0 \leq i \leq \floor{\frac{t-1}{2}}$. 
\end{enumerate}
\end{thm}

Again, nonzero $\OE_{2tn}$ characters are a product of characters, but this time there are $\floor{(t-1)/2}$ $\GL_{2n}$ characters, one $\OE_{2n}$ character and, if $t$ is even, one additional $\OO_{2n+1}$ character.
As in the symplectic factorization in \cref{thm:sympfact}, the even orthogonal character for $t = 2$ is nonzero if and only if $\core \lambda 2 = \emptyset$.
{Recall the involution (usually denoted $\omega$) on the space of symmetric functions the takes $s_\lambda$ to $s_{\lambda'}$. Koike and Terada have shown~\cite{koike-terada-1987} that this involution interchanges (universal) orthogonal characters and (universal) symplectic characters. Comparing \cref{thm:sympfact,thm:eorthfact}, it seems reasonable to suppose that we can obtain a proof of the latter from the former using this involution. However, this involution works only at the level of universal characters and does not commute with our specialization.
}

\begin{eg}
For $t = 2$, \cref{thm:eorthfact}  says that the character of the group $\OE(4)$ of the representation indexed by the partition $(a,b)$, $a \geq b \geq 0$, evaluated at  $(x, -x)$ is nonzero if and only if $a$ and $b$ have the same parity. 
If $a$ and $b$ are both odd, then 
\[
\oe_{(a,b)}(x, -x) = (-1)^{(b+1)/2} \oo_{(\frac{b-1}{2})}(-x^2) \oe_{(\frac{a+1}{2})}(x^2),
\]
and if $a$ and $b$ are both even, then 
\[
\oe_{(a,b)}(x, -x) = (-1)^{a/2} \oo_{(\frac{a}{2})}(-x^2) \oe_{(\frac{b}{2})}(x^2).
\]
Notice that all the characters on the right-hand side are for the groups $\OO(3)$ and $\OO(2)$, and in both cases the partitions indexing them are the $2$-quotients and of length 1.
\end{eg}

For the odd orthogonal case, we take $G = \OO_{2tn+1}$, the orthogonal group of $(2tn+1) \times (2tn+1)$ square matrices.
It will turn out that the notion of an `odd-orthogonal partition' is the same as being self-conjugate, or equivalently, $0$-asymmetric. 
The first few self-conjugate $3$-cores are $(1)$, $(3,1,1), (4,2,1,1)$ and $(6,4,2,2,1,1)$.
Our result for factorization of odd orthogonal characters is as follows.

\begin{thm}
\label{thm:oorthfact}
Let $\lambda$ be a partition of length at most $tn$ indexing an irreducible representation of $\OO_{2tn+1}$. Then the $\OO_{2tn+1}$ character $\oo_{\lambda}(X,\omega X, \dots ,\omega^{t-1}X)$ is given as follows.
\begin{enumerate}
    \item  If $\core \lambda t$ is not self-conjugate, then
\begin{equation}
    \label{th:3.1}
    \oo_{\lambda}(X,\omega X,\dots ,\omega^{t-1}X) = 0.
\end{equation}
\item If $\core \lambda t$ is self-conjugate with rank r, then
\begin{equation}
   \begin{split}
     \label{th:3.2} \oo_{\lambda}(X,\omega X, \dots ,\omega^{t-1}X)  =  (-1)^{\epsilon}\sgn(\sigma_{\lambda}) 
\prod_{i=0}^{\floor{\frac{t-2}{2}}} &  s_{\mu^{(3)}_i}(X^t,{\X}^t) \\ & 
\times 
\begin{cases}
\oo_{\lambda^{\left(\frac{t-1}{2}\right)}}(X^t) & t \text{ odd},\\
1 & t \text{ even},
\end{cases}   
    \end{split}
\end{equation}
where 
\[
\epsilon=- \ds \sum_{i=\floor{\frac{t}{2}}}^{t-1} \binom{n_{i}(\lambda)+1}{2} + \begin{cases}
nr & t \text{ odd},
\\ 0 & t \text{ even},
\end{cases}
\]
and 
$\ds \mu^{(3)}_i=  \lambda^{(t-1-i)}_1 + \left(\lambda^{(i)}, 0,\dots,0, -\rev(\lambda^{(t-1-i)})\right)$ has $2n$ parts for 
$0 \leq i \leq \floor{\frac{t-2}{2}}$.
\end{enumerate}
\end{thm}

Again, nonzero $\OO_{2tn+1}$ characters are a product of characters, but this time there are $\floor{t/2}$ $\GL_{2n}$ characters, and, if $t$ is odd, one additional $\OO_{2n+1}$ character. 
Since $2$-cores are always self-conjugate, odd orthogonal characters always have a nontrivial factorization when $t=2$.

\begin{eg}
For $t = 2$, \cref{thm:oorthfact}  says that the character of the group $\OO(5)$ of the representation indexed by the partition $(a,b)$, $a \geq b \geq 0$, evaluated at $(x, -x)$ is nonzero if and only if $a$ and $b$ have the same parity. We obtain
\[
\oo_{(a,b)}(x, -x) = (-1)^a s_{(\frac{a+b}{2},0)}(x^2,-x^2).
\]
Notice that the character on the right hand side is for $\GL(2)$ and involves the sum of the $2$-quotients.
\end{eg}

\begin{rem}
It might seem that the results of \cref{thm:schur-1}, \cref{thm:sympfact}, \cref{thm:eorthfact} and \cref{thm:oorthfact} are not well-defined because of \cref{rem:necklace}. More precisely, the lack of symmetry of the $t$-quotients on the right hand sides of these theorems might cause some worry. However, since changing $n \to n+1$ will change the length of the partition $\lambda$ by $tn$, the order of the quotients remains unchanged.
\end{rem}

\begin{rem}
In some cases, the Schur functions $s_{\mu_i^{(j)}}(X^t,{\X}^t)$ appearing on the right hand sides of \cref{thm:sympfact}, \cref{thm:eorthfact} and \cref{thm:oorthfact} for $j \in [3]$ respectively
factorize further into characters of other classical groups, but we do not understand this behavior fully. Whenever $\mu_i$ can be written as $\rho_1 + (\rho, -\rev (\rho))$ or $\rho_1 + (1 + \rho, -\rev (\rho))$ for a partition $\rho$ of length at most 
$n$, such a factorization occurs by the results in \cite{ayyer-behrend-2019}. In that case $s_{\mu_i}$ is either a product of two odd orthogonal characters or an even orthogonal and a symplectic character.
\end{rem}

It is natural to ask if there are infinitely many symplectic, orthogonal and self-conjugate $t$-cores. As we have seen, there are no symplectic or orthogonal $2$-cores and all $2$-cores are self-conjugate. For $t \geq 3$, it has been proved~\cite{garvan-kim-stanton-1990} that there are infinitely many self-conjugate $t$-cores. Our last result gives a generalisation.

\begin{thm} 
\label{thm:inf-cores}
There are infinitely many symplectic and orthogonal $t$-cores for $t \geq 3$.
\end{thm}

This is proved in \cref{sec:gf}.

\section{Background results}
\label{sec:background}

We collect all the assorted results we will need to prove our main results here. In \cref{sec:betasets}, we will use beta sets of partitions to classify symplectic partitions and their generalizations. We will derive generating functions for such partitions and prove that there are infinitely many of them in \cref{sec:gf}. Finally, we will derive determinant identities for block matrices in \cref{sec:det}.

\subsection{Properties of beta sets}
\label{sec:betasets}

In his treatise, Macdonald~\cite{macdonald-2015} used beta sets to derive powerful results for cores and quotients. We review and extend his results to the cases of interest. First, we recall a useful property of the beta numbers. 
Throughout, we will use the notation $[m] = \{1,\dots,m\}$ and $[m_1,m_2]=\{m_1,\dots,m_2\}$.

\begin{lem}
\label{lem:perm}
Let $\lambda$ and $\mu$ be partitions of length at most $m_1$ and $m_2$ respectively and let $m_2 \geq \lambda_1$. Then $\lambda'=\mu$ if and only if the $m_1+m_2$ numbers $\beta_j(\lambda)$ for $j \in [m_1]$ and $m_1+m_2-1-\beta_k(\mu)$ for $k \in [m_2]$ form a permutation of $\{0,1,\dots,m_1+m_2-1\}$.
\end{lem}

\begin{proof}
The forward implication holds by~\cite[Chapter I.1, Equation (1.7)]{macdonald-2015}. 

For the converse, since  $m_2 \geq \lambda_1$, the $m_1+m_2$ numbers $\beta_j(\lambda)$ for $j \in [m_1]$ and $m_1+m_2-1-\beta_k(\lambda')$ for $k \in [m_2]$ are a permutation of $\{0,1,\dots,m_1+m_2-1\}$ by~\cite[Chapter I.1, Equation (1.7)]{macdonald-2015}. 
So, $\beta_k(\lambda')=\beta_k(\mu)$, $k \in [m_2]$ and $\lambda'=\mu$.    
\end{proof}

Let $\lambda,\mu$ be partitions of length at most $m$ such that $\lambda \supset \mu$, and such that the set difference of Young diagrams $\lambda \setminus \mu$ is a border strip of length $t$. 
Then, it is known~\cite[Chapter I.1, Example 8(a)]{macdonald-2015} that $\beta(\mu)$ can be obtained from $\beta(\lambda)$ by subtracting $t$ from some part $\beta_i(\lambda)$ and rearranging in descending order. 
Therefore, for a partition $\lambda$ of length at most $m$,  we see that
\begin{equation}
\label{no-parts-partition=core}
 n_{i}(\lambda,m)=n_{i}(\core \lambda t, m), 
 \quad 0 \leq i \leq t-1.  
\end{equation}

We now 
explain the relationship between a partition and its conjugate in terms of their beta sets. 

\begin{lem}
\label{lem:beta sets la la'}
Let $\lambda$ and $\mu$ be partitions of length at most $tm_1$ and $tm_2$ respectively. 
If $\mu=\lambda'$, then
\begin{equation}
\label{z=0}
    n_{i}(\lambda)+n_{t-1-i}(\mu)=m_1+m_2,
    \quad 
    0 \leq i \leq t-1.
\end{equation}
The converse is true if $\lambda$ and $\mu$ are $t$-cores.
\end{lem}

\begin{proof}
Suppose $\mu=\lambda'$. 
Then \cref{lem:perm} implies that the numbers $\beta_j(\lambda)$ for $1 \leq j \leq tm_1$ and $tm+tn-1-\beta_k(\mu)$ for $1 \leq k \leq tm_2$
are a permutation of $\{0,1,\dots,tm_1+tm_2-1\}$.
Since $\xi \equiv t-1-i \pmod{t}$ 
implies $tm_1+tm_2-1-\xi \equiv i \pmod{t}$, $n_{i}(\lambda)+n_{t-1-i}(\mu)$ 
is equal to number of integers in $\{0,1,\dots,tm_1+tm_2-1\}$ 
congruent to $i \pmod{t}$.
Since for each $0 \leq i \leq t-1$, there are $m_1+m_2$ numbers in $\{0,1,\dots,tm_1+tm_2-1\}$
congruent to $i$ modulo $t$, \eqref{z=0} holds.

Conversely, assume $\lambda$ and $\mu$ are $t$-cores and \eqref{z=0} holds. 
Fix $i$, $0 \leq i \leq t-1$.
Since $\lambda$ is a $t$-core, the numbers $i < i+t < \dots < i+(n_{i}(\lambda)-1)t$ 
occur in $\beta(\lambda)$. Similarly, since $\mu$ is a $t$-core and $0 \leq t-i-1 \leq t-1$, the numbers $t-1-i < 2t-1-i < \dots < (n_{t-1-i}(\mu))t-1-i$ occur in $\beta(\mu)$. 

So, the parts of $\beta(\lambda)$ and $tm_1+tm_2-1-\beta(\mu)$ congruent to $i \pmod t$ are 
\[
i < i+t < \dots < i+(n_{i}(\lambda)-1)t 
\] 
and
\[
t(m_1+m_2-1)+i>t(m_1+m_2-2)+i>\dots>(n_{i}(\lambda))t+i
\] 
respectively. Therefore, all the numbers congruent to $i \pmod t$ between $i$ and $t(m_1+m_2-1)+i$ appear in the union of $\beta(\lambda)$ and $tm_1+tm_2-1-\beta(\mu)$.
Since this holds for each $0 \leq i \leq t-1$, parts of $\beta(\lambda)$ and $tm_1+tm_2-1-\beta(\mu)$ are a permutation of $\{0,1,\dots,tm_1+tm_2-1\}$. 
Moreover, the largest part of $\beta(\lambda)$ is at most $t(m_1+m_2)-1$ implies $\lambda_1$ is at most $tm_2$. So, by \cref{lem:perm}, $\mu=\lambda'$, completing the proof.
\end{proof}

Using \cref{lem:beta sets la la'} and \eqref{no-parts-partition=core} for $m=tn$,  we have the following corollary.

\begin{cor} 
\label{cor:con}
For a partition $\lambda$ of length at most $tn$,
$\core \lambda t$ is self-conjugate if and only if
\begin{equation}
n_{i}(\lambda)+n_{t-1-i}(\lambda)=2n, 
\quad 0 \leq i \leq \floor{\frac{t-1}{2}}.   
\end{equation}
\end{cor}

Recall the definition of $z$-asymmetric partition from \cref{def:z-asym}.
Let $\mathcal{P}_{z}$ be the set of $z$-asymmetric partitions and $\mathcal{P}_{z,t}$ be the set of $z$-asymmetric $t$-cores.

\begin{lem}
\label{lem:cond} 
Let $\lambda=(\alpha|\beta)$ be a partition of length at most $m$ and rank $r$.
Then the following statements are equivalent.
\begin{enumerate}
    \item \label{condd1} $\lambda \in \mathcal{P}_{z}$.
    \item \label{condd}
an integer $\xi$ between  $0$ and  $m-z-1$ occurs in $\beta(\lambda)$ if and only if  $2m-z-1-\xi$ does not.
\item \label{condd2}
$\beta(\lambda)$ is obtained from the sequence $(\alpha_1+m,\dots,\alpha_r+m, m-1,\dots,1,0$)
by deleting the numbers $m-z-1-\alpha_r>m-z-1-\alpha_{r-1}>\dots>m-z-1-\alpha_1$ 
lying between $0$ and $m-1$.
\end{enumerate}
\end{lem}

\begin{proof} First, note that $\lambda \in \mathcal{P}_{z}$ if and only if $\lambda$ is of the form
\begin{equation*}
 \lambda=  (\alpha_1+1,\dots,\alpha_r+r,
 \underbrace{r,\dots,r}_{\alpha_r+z },
 \underbrace{r-1,\dots,r-1}_{\alpha_{r-1}-\alpha_r-1 },\dots,
 \underbrace{1,\dots,1}_{\alpha_1-\alpha_2-1}).
\end{equation*} 
In that case, its beta set is 
\begin{multline*}
\beta(\lambda)=(\alpha_1+m,\dots,\alpha_r+m,
\underbrace{m-1,\dots,m-(\alpha_r+z)}_{\alpha_r+z },
\widehat{m-(\alpha_r+z+1)},\\
\underbrace{m-(\alpha_r+z+2),
\dots,m-(\alpha_{r-1}+z)}_{\alpha_{r-1}-\alpha_r-1 },
\widehat{m-(\alpha_{r-1}+z+1)},
\dots,\widehat{m-(\alpha_2+z+1)},\\
\underbrace{m-(\alpha_2+z+2),\dots,m-(\alpha_1+z)}_{\alpha_1-\alpha_2-1},\widehat{m-(\alpha_1+z+1)},m-(\alpha_1+z+2),\dots,0
),
\end{multline*}
where a hat on an entry denotes its absence from the tuple.
So, \cref{condd1} and \cref{condd2} are equivalent. 
See \cref{fig:z-asymmetric}(a) and (b) for the last few rows of the Young diagram of a $z$-asymmetric partition and its beta set.

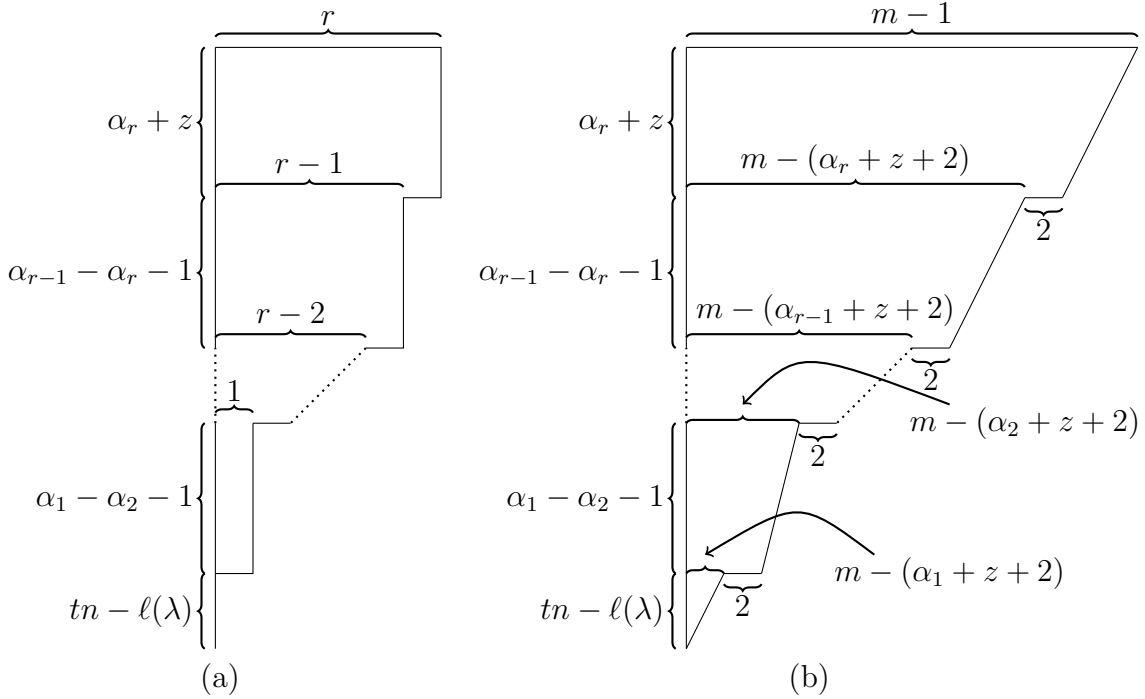
\begin{figure}[!htbp]
\begin{center}

\begin{tabular}{c c}
     \begin{tikzpicture}
     \draw (-1,0)--(2,0)--(2,-2)--(1.5,-2)--(1.5,-4)--(1,-4);
     \draw (-1,-7)--(-1,-5);
\draw (-1,-4)--(-1,0);
\draw[dotted,thick,xshift=0pt,yshift=0pt] (-1,-5)--(-1,-4);
\draw[dotted,thick,xshift=0pt,yshift=0pt] (1,-4)--(0,-5);
\draw[decorate,decoration={brace},thick,xshift=-4pt,yshift=0pt] (-1,-2)--(-1,0) node[midway,left]{$\alpha_r+z$};
 \draw[decorate,decoration={brace},thick,xshift=-4pt,yshift=0pt] (-1,-4)--(-1,-2) node[midway,left]{$\alpha_{r-1}-\alpha_r-1$};
 \draw[decorate,decoration={brace},thick,xshift=-4pt,yshift=0pt] (-1,-7)--(-1,-5) node[midway,left]{$\alpha_1-\alpha_2-1$};
 \draw (-0.5,-5)--(0,-5);
\draw (-0.5,-5)--(-0.5,-7)--(-1,-7);
\draw[decorate,decoration={brace},thick,xshift=0pt,yshift=4pt] (-1,0)--(2,0) node[midway,above]{$r$};
\draw[decorate,decoration={brace},thick,xshift=0pt,yshift=4pt] (-1,-2)--(1.5,-2) node[midway,above]{$r-1$};
\draw[decorate,decoration={brace},thick,xshift=0pt,yshift=4pt] (-1,-4)--(1,-4) node[midway,above]{$r-2$};
\draw[decorate,decoration={brace},thick,xshift=0pt,yshift=4pt] (-1,-5)--(-0.5,-5) node[midway,above]{$1$};
\draw (-1,-8)--(-1,-7);
\draw[decorate,decoration={brace},thick,xshift=-4pt,yshift=0pt] (-1,-8)--(-1,-7) node[midway,left]{$tn-\ell(\lambda)$};
     \end{tikzpicture}
&
\hfill     \begin{tikzpicture}
\draw (-1,-7)--(-1,-5);
\draw (-1,-4)--(-1,0);
\draw (-1,0)--(5,0);
 \draw (5,0)--(4,-2);
 \draw (4,-2)--(3.5,-2);
 \draw (3.5,-2)--(2.5,-4);
 \draw (2.5,-4)--(2,-4);
 \draw[dotted,thick,xshift=0pt,yshift=0pt] (2,-4)--(1,-5);
 \draw (1,-5)--(0.5,-5);
 \draw (0.5,-5)--(0,-7);
 \draw (0,-7)--(-0.5,-7);
  \draw (-0.5,-7)--(-1,-8);
 \draw (-1,-8)--(-1,-7);
\draw[dotted,thick,xshift=0pt,yshift=0pt] (-1,-5)--(-1,-4);
\draw[decorate,decoration={brace},thick,xshift=0pt,yshift=4pt] (-1,-2)--(3.5,-2) node[midway,above]{$m-(\alpha_r+z+2)$};
\draw[decorate,decoration={brace},thick,xshift=0pt,yshift=0pt] (-1,-5)--(0.5,-5);
\node (A) at (3.5, -5) {$m-(\alpha_2+z+2)$};
\draw[thick, ->] (2.5, -4.75) .. controls (0.5,-4) .. (-0.25, -4.75);
\draw[decorate,decoration={brace},thick,xshift=0pt,yshift=-4pt] (1,-5)--(0.5,-5) node[midway,below]{$2$};
\draw[decorate,decoration={brace},thick,xshift=0pt,yshift=-4pt] (4,-2)--(3.5,-2) node[midway,below]{$2$};
\draw[decorate,decoration={brace},thick,xshift=0pt,yshift=-4pt] (2.5,-4)--(2,-4) node[midway,below]{$2$};
\draw[decorate,decoration={brace},thick,xshift=0pt,yshift=-4pt] (0,-7)--(-0.5,-7) node[midway,below]{$2$};
\draw[decorate,decoration={brace},thick,xshift=0pt,yshift=4pt] (-1,-4)--(2,-4) node[midway,above,xshift=10pt]{$m-(\alpha_{r-1}+z+2)$};
\draw[decorate,decoration={brace},thick,xshift=0pt,yshift=4pt] (-1,0)--(5,0) node[midway,above]{$m-1$};
 \draw[decorate,decoration={brace},thick,xshift=-4pt,yshift=0pt] (-1,-2)--(-1,0) node[midway,left]{$\alpha_r+z$};
 \draw[decorate,decoration={brace},thick,xshift=-4pt,yshift=0pt] (-1,-4)--(-1,-2) node[midway,left]{$\alpha_{r-1}-\alpha_r-1$};
 \draw[decorate,decoration={brace},thick,xshift=-4pt,yshift=0pt] (-1,-7)--(-1,-5) node[midway,left]{$\alpha_1-\alpha_2-1$};
 \draw[decorate,decoration={brace},thick,xshift=0pt,yshift=0pt] (-1,-7)--(-0.5,-7);
 \node (A) at (2.5, -7) {$m-(\alpha_1+z+2)$};
 \draw[thick, ->] (1.5, -6.75) .. controls (0.5,-6) .. (-0.75, -6.75);
 \draw[decorate,decoration={brace},thick,xshift=-4pt,yshift=0pt] (-1,-8)--(-1,-7) node[midway,left]{$tn-\ell(\lambda)$};
\end{tikzpicture}
\\
(a) & (b)
\end{tabular}

\caption{Rows $r+1$ through $tn$ a $z$-asymmetric partition $(\alpha | \, \alpha + z)$ of rank $r$ on the left and the same rows of its beta set on the right. }
\label{fig:z-asymmetric}
\end{center}
\end{figure}

Clearly, \cref{condd2} implies \cref{condd}. 
Now suppose \cref{condd} holds. Observe that a part of $\beta(\lambda)$, $\lambda_i+m-i$ is greater than and equal to $m$ if and only if $\lambda_i$ is greater than and equal to $i$.
Thus there are $r$ parts of $\beta(\lambda)$ greater than m. Since $\alpha_1+m > \dots > \alpha_r+m$ are $r$ integers greater than and equal to $m$ which occur in $\beta(\lambda)$, \cref{condd2} holds.
\end{proof}

\begin{lem}
\label{lem:z-large}
For $2 \leq t \leq z + 1$, the empty partition is the only $t$-core in $\mathcal{P}_{z,t}$. 
\end{lem}

\begin{proof}
Let $\lambda=(\alpha|\beta) \in \mathcal{P}_{z}$ have rank $r > 0$.
Then 
\begin{equation*}
 \lambda=  (\alpha_1+1,\dots,\alpha_r+r,
 \underbrace{r,\dots,r}_{\alpha_r+z },
 \underbrace{r-1,\dots,r-1}_{\alpha_{r-1}-\alpha_r-1 },\dots,
 \underbrace{1,\dots,1}_{\alpha_1-\alpha_2-1}).
\end{equation*}
So, $\lambda_{r+i}=r$, $1 \leq i \leq \alpha_r+z$ and $\lambda'_r=\alpha_r+r+z$.
If $z \geq t-1$, then $\lambda_{r+\alpha_r+z-t+1}=r$. Hence, the hook number $h(r+\alpha_r+z-t+1,r)=(r)+(\alpha_r+r+z)-(r+\alpha_r+z-t+1)-(r)+1=t$, which is a contradiction since $\lambda$ is a $t$-core. So, $\lambda$ must be empty.
\end{proof}

Now we explain the constraints satisfied by $n_{i}(\lambda), 0 \leq i \leq t-1$, for a $z$-asymmetric $t$-core $\lambda$ of length at most $tn$.

\begin{lem} 
\label{lem:converse: sym}
Let $\lambda$ be a $t$-core of length at most $tn$ and $0 \leq z \leq t-2$. Then $\lambda \in \mathcal{P}_{z,t}$ if and only if
\begin{equation}
\begin{split}
  \label{n_{t-1}} 
    n_{i}(\lambda)+n_{t-z-1-i}(\lambda) =& 2n
    \quad \text{for} \quad  0 \leq i \leq  t-z-1, \\ 
    \text{and} \quad n_{i}(\lambda) =& n, \quad t-z \leq i \leq t-1.
\end{split}
\end{equation}
\end{lem}

\begin{proof} 
Suppose $\lambda=(\alpha|\alpha+z)$ and $\rk(\lambda)=r$.
Using \cref{lem:cond}(3),   $\beta(\lambda)$ is obtained from the sequence $(\alpha_1+tn,\dots,\alpha_r+tn, tn-1,\dots,1,0$) by deleting the numbers 
$tn-z-1-\alpha_r > tn-z-1-\alpha_{r-1} > \dots > tn-z-1-\alpha_1$. 
Since $n_{i}(\emptyset,tn) = n$ for all $i$, \eqref{n_{t-1}} trivially holds for the empty partition.
Note that if $tn-z-1-\alpha_i \equiv \theta_i \pmod{t}$, then $\alpha_i+tn  \equiv t-z-1-\theta_i \pmod{t}$ for $i \in [r]$. In that case $n_{t-z-1-\theta_i}(\lambda)$ increases by one and $n_{\theta_i}(\lambda)$ decreases by one.
Therefore, it is sufficient to show that $\theta_i \in [0,t-z-1]$, for each $i \in [r]$ to prove \eqref{n_{t-1}}.

{We prove this successively in reverse order starting from $\theta_r$ and going all the way to $\theta_1$.} Since $\lambda$ is a $t$-core, if $tn-z-1-\alpha_r$ does not occur in $\beta(\lambda)$, then neither does $tn-z-1-\alpha_r+t$. 
Since $tn-z-1-\alpha_r$ is the largest number deleted  from $(tn-1,tn-2,\dots,0)$ to get $\beta(\lambda)$, $tn-z-1-\alpha_r+t \geq tn$. So,
$\alpha_r+z+1 \in [z+1,t]$; and $\theta_r \in [0,t-z-1]$.
There is nothing to show if $\theta_{r-1}=\theta_r$. So, 
assume $\theta_{r-1} \neq \theta_{r}$. Similarly, since $\lambda$ is a $t$-core, if $tn-z-1-\alpha_{r-1}$ does not occur in $\beta(\lambda)$, then neither does $tn-z-1-\alpha_{r-1}+t$. 
Since $tn-z-1-\alpha_{r-1}$ is the largest number congruent to $\theta_{r-1}$ deleted from $(tn-1,tn-2,\dots,0)$ to get $\beta(\lambda)$, 
$\alpha_{r-1}+z+1 \in [z+1,t]$ and $\theta_{r-1} \in [0,t-z-1]$.
Proceeding in this manner, $\theta_i \in [0,t-z-1]$ for all $i \in [r]$. 

Conversely, assume \eqref{n_{t-1}} holds for $\lambda$. 
If $\lambda$ is the empty partition, then it belongs to $\mathcal{P}_{z,t}$ vacuously. Now suppose $\lambda$ is non-empty and  $\{i_1,i_2,\dots,i_k\}_{>} 
\subset \{0,1,\dots,t-z-1\}$ 
such that $n_{i_j}(\lambda)>n$ which implies $n_{t-z-1-i_j}(\lambda)<n, j \in [k]$. 
Since $\lambda$ is a $t-$core, for each $j$,
$i_j+tn<i_j+t(n+1)<\dots<i_j+t(n_{i_j}(\lambda)-1)$
are the parts of  $\beta(\lambda)$ greater than and equal to $tn$.
If $n_{t-z-1-i_j}(\lambda)<n$ for $j \in [k]$ implies parts of $\beta(\lambda)$ less than and equal to $tn-1$ is obtained from the sequence $(tn-1,tn-2,\dots,0)$ by deleting the numbers
\[
tn-z-1-i_j, t(n-1)-z-1-i_j, \dots, t (n_{t-2-i_j}(\lambda)+1)-z-1-i_j.
\]
 Observe that an integer $\xi$ between $0$ and $tn-z-1$ occurs in $\beta(\lambda)$ if and only if $2tn-z-1-\xi$ does not. 
   So, by \cref{lem:cond}, $ \lambda \in \mathcal{P}_{z,t}$.
\end{proof}

\begin{cor}
\label{cor:symp-parts}
Let $t \geq 3$ and  $\lambda$ be a partition of length 
at most $tn$. Then $\core \lambda t$ is a symplectic $t$-core if and only if
$n_{i}(\lambda)+n_{t-2-i}(\lambda)=2n$ for $0 \leq i \leq \floor{\frac{t-2}{2}}$ and $n_{t-1}(\lambda)=n$.   
\end{cor}

\begin{proof} 
Set $z = 1$ in \cref{lem:converse: sym}.
This now follows using $\ell(\core \lambda t) \leq \ell(\lambda) \leq tn$ 
and \eqref{no-parts-partition=core} for $m=tn$.
\end{proof}

Since $\core \lambda {t}'=\core {\lambda'} t$~\cite[Example I.1(e)]{macdonald-2015}, it follows that $\core \lambda t$ is an orthogonal $t$-core 
if and only if $\core {\lambda'} t$ is a symplectic $t$-core. We then have the following corollary.

\begin{cor}
\label{cor:ocore}
Let $\lambda$ be a partition of length 
at most $tn$.
Then core$_t(\lambda)$ is an orthogonal $t$-core if and only if $n_{0}(\lambda)=n \text{ and } n_{i}(\lambda)+n_{t-i}(\lambda)=2n \text{ for } 1 \leq i \leq \floor{\frac{t}{2}}$.
\end{cor}

\begin{proof}
Suppose $\ell(\lambda') \leq tm$, for some $m \geq 1$. 
Using \cref{cor:symp-parts} for $\lambda'$, $\core \lambda t$ is an orthogonal $t$-core if and only if 
$n_{t-1}(\lambda')=m$ and $n_{i}(\lambda')+n_{t-2-i}(\lambda')=2m$ for $0 \leq i \leq \floor{\frac{t-2}{2}}$.
Now using \cref{lem:beta sets la la'}, we get
the desired result.
\end{proof}

For completeness, we note the following property of the $t$-quotient of orthogonal and symplectic partitions, although we will not use it.

\begin{prop}[{\cite[Bijection 3]{garvan-kim-stanton-1990}}]
\label{prop:sym-iff} 
Let $\lambda$ be a partition. If
\begin{enumerate}
\item $\lambda^{(0)}$ is an orthogonal partition,
\item $\core \lambda t$ is an orthogonal $t$-core, and 
\item  $(\lambda^{(i)})'=\lambda^{(t-i)}$ for $\quad 1 \leq i \leq \floor{\frac{t}{2}}$,
\end{enumerate}
then $\lambda$ is orthogonal. A similar statement holds for symplectic partitions.
\end{prop}

We now see how to compute the rank of a $t$-core from its beta-set.

\begin{lem}
\label{lem:rank-core}
If $\lambda$ is a $t$-core of length at most $tn$, then  
\begin{equation}
\label{rnk}
 \rk(\lambda) =\sum_{i=0}^{t-1}(n_{i}(\lambda)-n)_+,
\end{equation}
where $z_+ := \max(z,0)$.
\end{lem}

\begin{proof}
If $n_{i}(\lambda) = n$ for $0 \leq i \leq t-1$, then $\beta(\lambda)=(
tn-1,tn-2,\dots,1,0)$ which implies $\lambda$ is an empty partition. So, the result holds in this case.
Otherwise, assume $\{i_1,i_2,\dots,i_k\}_{>} \subset \{0,1,\dots,t-1\}$ such that $n_{i_j}(\lambda)>n \text{ for } 1 \leq j \leq k$.  
Since $\lambda$ is a $t$-core, 
\[
i_j+tn<i_j+t(n+1)<\dots<i_j+t(n_{i_j}(\lambda)-1)
\]
are the parts of $\beta(\lambda)$ greater than $tn-1$ for each $j$.  
If $r$ is the number of parts of  $\beta(\lambda)$ greater than  $tn-1$, then
\[ 
r= \sum_{j=1}^k(n_{i_j}(\lambda)-n)=\sum_{i=0}^{t-1}(n_{i}(\lambda)-n)_+.
\]
Moreover, $\beta_{r}(\lambda)$ is the smallest part of  $\beta(\lambda)$ greater than  $tn-1$ and is therefore equal to $i_k+tn$.
So, $\ds \lambda_{r}=\beta_{r}(\lambda)-(tn-r)=tn+i_k-(tn-r)=i_k+r \geq r$ and $\lambda_{r+1} \leq tn-1-(tn-r-1) \leq r$, which implies the rank of $\lambda$ is $r$.
\end{proof}

\cref{lem:rank-core} immediately tells us how to compute the rank of the $t$-core of a partition using \eqref{no-parts-partition=core}.

\begin{cor}
\label{cor:rank-tcore}
If $\lambda$ is a partition of length at most $tn$, then
\[
\rk(\core{\lambda}t)= \sum_{i=0}^{t-1}(n_{i}(\lambda)-n)_+.
\]
\end{cor}

\cref{lem:rank-core} also gives us an algorithm to determine if a partition has empty $t$-core.

\begin{cor}
\label{cor:ary}
If $\lambda$ is a partition of length at most $tn$, then $\core \lambda t$ is {empty} if and only if $n_{i}(\lambda)=n$ for $0 \leq i \leq t-1$. 
\end{cor}

\begin{lem}
\label{lem:rank-sympcore-etc}
Let $\lambda$ be a partition of length 
at most $tn$.
\begin{enumerate}

\item If $\core \lambda t$ is a symplectic $t$-core, then 
\begin{equation}
\label{score1}
 \rk(\core \lambda t)=\ds \sum_{i=0}^{\floor{\frac{t-3}{2}}}|n_{i}(\lambda)-n|= \sum_{i=\floor{\frac{t-1}{2}}}^{t-2}|n_{i}(\lambda)-n|.  
\end{equation}

\item If $\core \lambda t$ is an orthogonal $t$-core, then
\begin{equation}
\label{ocore1}
\rk(\core \lambda t) = \sum_{i=1}^{\floor{\frac{t-1}{2}}}|n_{i}(\lambda)-n|. 
\end{equation}

\item If $\core \lambda t$ is self-conjugate $t$-core, then 
\begin{equation}
\label{con1} 
\rk(\core \lambda t)=\ds \sum_{i=0}^{\floor{\frac{t-2}{2}}}|n_{i}(\lambda)-n|.
\end{equation}

\end{enumerate}
\end{lem}

\begin{proof} Using \cref{cor:rank-tcore}, 
\begin{equation*}
 \rk(\text{core}_t(\lambda)) = \sum_{i = 0}^{t-1} (n_{i}(\lambda)-n)_+. \end{equation*}
If $\core \lambda t$ is a symplectic $t$-core, then by \cref{cor:symp-parts}, 
\begin{equation*}
n_{t-1}(\lambda)=n \text{ and }
    n_{i}(\lambda)+n_{t-2-i}(\lambda)=2n \text{ for } 0 \leq i \leq \floor{\frac{t-2}{2}}.
\end{equation*}
If $n_{i}(\lambda)>n$ for some $i \in \{\floor{\frac{t-3}{2}}+1,\floor{\frac{t-3}{2}}+2,\dots,t-2\}$, then $n_{t-2-i}(\lambda)<n$ and $n_{i}(\lambda)-n=n-n_{t-2-i}(\lambda)$. Since $t-2-i \in \{0,1,\dots,\floor{\frac{t-3}{2}}\}$,
\[
\rk(\text{core}_t(\lambda))=\ds \sum_{p=0}^{\floor{\frac{t-3}{2}}}|n_{p}(\lambda)-n|.
\]
Using an argument analogous to the one just given as well as \cref{cor:ocore} and \cref{cor:con}, the proofs of \eqref{ocore1} and \eqref{con1} follow.
\end{proof}

\subsection{Determinant evaluations}
\label{sec:det}

Here, we will derive all the determinant evaluations we need to prove our character identities. We will state them in the most general form possible.

Let $\lambda$ be a partition with $\ell(\lambda) \leq tn$.
Recall for $0 \leq p \leq t-1$,
$\beta_j^{(p)}(\lambda), 
1 \leq j \leq n_{p}(\lambda)$
are the parts of $\beta(\lambda)$ congruent to 
$p$ modulo $t$, arranged in decreasing order.
In addition, for $q \in \mathbb{Z} \cup (\mathbb{Z}+1/2)$,
define $n \times n_{p}(\lambda)$ matrices
\begin{equation}
\label{def AB}
A^{\lambda}_{p,q} 
=\left(
{x}_i^{\beta^{(p)}_j({\lambda})+q}
\right)_{\substack{1 \leq i \leq n\\1 \leq j \leq n_{p}(\lambda)}}, 
\quad 
\bar{A}^{\lambda}_{p,q}
= \left(
\bar{x}_i^{\beta_j^{(p)}(\lambda)+q}
\right)_{\substack{1 \leq i\leq n \\ 1\leq j\leq n_{p}(\lambda)}}.
\end{equation}
The corresponding matrices for the empty partition are denoted by
\begin{equation}
\label{def AB0}
A_{p,q}= \left(
x_i^{t(n-j)+p+q}
\right)_{1 \leq i,j\leq n }, 
\quad 
\bar{A}_{p,q}=\left(
\bar{x}_i^{t(n-j)+p+q}
\right)_{1 \leq i,j\leq n }.
\end{equation}
In all cases, whenever $q=0$, we will omit it. For example, we will write $A^{\lambda}_{p}$ instead of $A^{\lambda}_{p,0}$.
Recall that the $t$-quotient of $\lambda$ is given by
$\quo \lambda t = (\lambda^{(0)},\dots, \lambda^{(t-1)})$
and $n_{p}(\lambda) \leq n$ for $0 \leq p \leq t-1$. 
Then, using \cref{prop:mcd-t-core-quo}(2),
\[
t\beta_j(\lambda^{(p)})=\beta_j^{(p)}(\lambda)-p, \quad 1\leq j\leq n,
\]
we write down alternate formulas for the classical characters.
Recall that $X^t = (x_1^t,\dots,x_n^t)$.
Using this notation, {we see that} the Schur polynomial is given by
\begin{equation}
\label{snew}
s_{\lambda^{(p)}}(X^t)
  = \frac{\det A^{\lambda}_{p}}
  {\det A_{p}},
  \end{equation}
the symplectic character is given by
\begin{equation}
\label{sp-new}
\sp_{\lambda^{(p)}}(X^t)
  = \frac{\det \left(A^{\lambda}_{p,t-p}-\bar{A}^{\lambda}_{p,t-p}\right)}
  {\det \left(A_{p,t-p}-\bar{A}_{p,t-p}\right)},
  \end{equation}
the odd orthogonal character is given by
\begin{equation}
\label{oo-new}
 \oo_{\lambda^{(p)}}(X^t)
 =\frac{\det \left(A^{\lambda}_{p,\frac{t}{2}-p}-\bar{A}^{\lambda}_{p,\frac{t}{2}-p}\right)}
 {\det \left(A_{p,\frac{t}{2}-p}-\bar{A}_{p,\frac{t}{2}-p}\right)},   
\end{equation}
and the even orthogonal character is given by
\begin{equation}
\label{eo-new}
  \oe_{\lambda^{(p)}}(X^t)
  =\frac{ {2} \det \left(A^{\lambda}_{p,-p}+\bar{A}^{\lambda}_{p,-p}\right)}
  {{(1+\delta_{\lambda_n^{(p)},0})}\det \left(A_{p,-p}+\bar{A}_{p,-p}\right)},
\end{equation}
using \eqref{spdef}, \eqref{oodef} and \eqref{oedef} respectively.

We first express the Schur function of partitions of length at most $2n$ in
the variables $X^t \cup X^{-t}$ occuring in our theorems in this notation.

\begin{lem}
\label{lem:s-new}
Let $\lambda$ be a partition of length at most $tn$ with 
$\quo \lambda t = (\lambda^{(0)},\dots,\lambda^{(t-1)})$. 
If $p,q \in \{0,1,\dots,t-1\}$ such that 
$n_{p}(\lambda)+n_{q}(\lambda)=2n$, then we define
$\rho_{p,q}=\lambda^{(p)}_1 + (\lambda^{(q)}, 0,\dots,0, -\rev(\lambda^{(p)}))$,
where we pad $0'$s in the middle so that $\rho_{p,q}$ is of length $2n$. 
Then the Schur function $s_{\rho_{p,q}}(X^t,{\X}^t)$ can be written as
\begin{equation}
\label{lemmaa}
 s_{\rho_{p,q}}(X^t,{\X}^t)
 = \frac{(-1)^{\frac{n_{p}(\lambda)(n_{p}(\lambda)-1)}{2}} }
 {(-1)^{\frac{n(n-1)}{2}} }
\frac{  \det \left( \begin{array}{c|c}
   A^{\lambda}_{q,-q}  & \bar{A}^{\lambda}_{p,t-p} \\[0.2cm]
   \hline \\[-0.3cm]
   \bar{A}^{\lambda}_{q,-q}  & A^{\lambda}_{p,t-p}
\end{array}\right)}
{\det \left( \begin{array}{c|c}
   A_{q,-q}  & \bar{A}_{p,t-p} \\[0.2cm]
   \hline \\[-0.3cm]
   \bar{A}_{q,-q}  & A_{p,t-p}
\end{array}\right)}.
\end{equation}
\end{lem}

\begin{proof}
We will think of the first $n_{q}(\lambda)$ components of $\rho_{p,q}$ as coming from $\lambda^{(q)}$ and the remaining as coming from $\lambda^{(p)}$.
Using the Schur polynomial expression \eqref{gldef}, {we see that}
the numerator of $s_{\rho_{p,q}}(X^t,{\X}^t)$ is
\[
\det \left(\begin{array}{c|c}
  \left(
  x_i^{t(\lambda^{(p)}_1+\lambda^{(q)}_j+2n-j)}
  \right)_{\substack{1 \leq i \leq n\\1 \leq j \leq n_{q}(\lambda)}}  
  & \left(
  x_i^{t(\lambda^{(p)}_1-\lambda^{(p)}_{2n+1-j}+2n-j)}
  \right)_{\substack{1 \leq i \leq n\\ n_{q}(\lambda)+1 \leq j \leq 2n}}  \\\\\hline\\
   \left(\bar{x}_i^{t(\lambda^{(p)}_1+\lambda^{(q)}_j+2n-j)}\right)_{\substack{1 \leq i \leq n\\1 \leq j \leq n_{q}(\lambda)}}  & \left(\bar{x}_i^{t(\lambda^{(p)}_1-\lambda^{(p)}_{2n+1-j}+2n-j)}\right)_{\substack{1 \leq i \leq n\\n_{q}(\lambda)+1 \leq j \leq 2n}} 
\end{array}
\right). \]
Multiplying row $i$ in the top blocks and bottom blocks of the numerator by  $\bar{x}_i^{t(\lambda_1^{(p)}+n_{p}(\lambda))}$ 
and 
${x}_i^{t(\lambda_1^{(p)}+n_{p}(\lambda))}$ respectively, 
for each $i=1,2,\dots,n$ and then reversing the last $n_{p}(\lambda)$ columns, {we see that} the numerator equals
\begin{equation}
    \begin{split}
     \label{num-s}
(-1)^{\frac{n_{p}(\lambda)(n_{p}(\lambda)-1)}{2}} \det \left(
\begin{array}{c|c}
  \left(
  x_i^{\beta^{(q)}_j(\lambda)-q}
  \right)_{\substack{1 \leq i \leq n\\1 \leq j \leq n_{q}(\lambda)}}   & \left(
  \bar{x}_i^{\beta^{(p)}_j(\lambda)-p+t}
  \right)_{\substack{1 \leq i \leq n\\1 \leq j \leq n_{p}(\lambda)}}  
  \\\\\hline\\
   \left(
   \bar{x}_i^{\beta^{(q)}_j(\lambda)-q}
   \right)_{\substack{1 \leq i \leq n\\1 \leq j \leq n_{q}(\lambda)}}  & \left(
   {x}_i^{\beta^{(p)}_j(\lambda)-p+t}
   \right)_{\substack{1 \leq i \leq n\\1 \leq j \leq n_{p}(\lambda)}} 
\end{array}
\right) \\ 
=  (-1)^{\frac{n_{p}(\lambda)(n_{p}(\lambda)-1)}{2}}
\det \left( \begin{array}{c|c}
   A^{\lambda}_{q,-q}  & \bar{A}^{\lambda}_{p,t-p} \\[0.2cm]
   \hline \\[-0.3cm]
   \bar{A}^{\lambda}_{q,-q}  & A^{\lambda}_{p,t-p}
\end{array}\right).     
    \end{split}
\end{equation}
Since $n_{p}(\emptyset,tn)) = n_{q}(\emptyset,tn)) = n$ and the denominator in the expression \eqref{gldef} is the same as its numerator evaluated at the empty partition, we see that the denominator is
\[
(-1)^{\frac{n(n-1)}{2}}  \det \left( \begin{array}{c|c}
   A^{\empty}_{q,-q}  & \bar{A}_{p,t-p} \\[0.2cm]
   \hline \\[-0.3cm]
   \bar{A}_{q,-q}  & A_{p,t-p}
\end{array}\right).
\]
Hence, \eqref{lemmaa} holds.  
\end{proof}

The next result shows that the role of $p$ and $q$ in these kind of Schur evaluations can be interchanged.

\begin{lem} 
Using the same notation as in \cref{lem:s-new}, {we see that}
\[ 
s_{\rho_{p,q}}(X,{\X})=s_{\rho_{q,p}}(X,{\X}).
\]
\end{lem}

\begin{proof} 
Since ${\x}_i^{t} A^{\lambda}_{p,t-p} = A^{\lambda}_{p,-p}$ and ${x}_i^{t} \bar{A}^{\lambda}_{p,t-p} = \bar{A}^{\lambda}_{p,-p}$, we observe
\[
\left( \begin{array}{c|c}
   0 & {\x}_i^{t} I_n  \\
   \hline
{x}_i^{t} I_{n}  & 0
\end{array}\right) 
\left( \begin{array}{c|c}
   A^{\lambda}_{q,-q}  & \bar{A}^{\lambda}_{p,t-p} \\[0.2cm]
   \hline \\[-0.3cm]
   \bar{A}^{\lambda}_{q,-q}  & A^{\lambda}_{p,t-p}
\end{array}\right)
\left(\begin{array}{c|c}
   0  & I_{n_q(\lambda)} \\
   \hline
    I_{n_p(\lambda)}  & 0
\end{array}\right) = \left( \begin{array}{c|c}
   A^{\lambda}_{p,-p}  & \bar{A}^{\lambda}_{q,t-q} \\[0.2cm]
   \hline \\[-0.3cm]
   \bar{A}^{\lambda}_{p,-p}  & A^{\lambda}_{q,t-q}
\end{array}\right),
\]
where $I_m$ is the $m \times m$ identity matrix.
Evaluating the determinant on both sides, 
\[
(-1)^{n^2} \det \left( \begin{array}{c|c}
   A^{\lambda}_{q,-q}  & \bar{A}^{\lambda}_{p,t-p} \\[0.2cm]
   \hline \\[-0.3cm]
   \bar{A}^{\lambda}_{q,-q}  & A^{\lambda}_{p,t-p}
\end{array}\right) (-1)^{n_p(\lambda) n_q(\lambda)}=\det \left( \begin{array}{c|c}
   A^{\lambda}_{p,-p}  & \bar{A}^{\lambda}_{q,t-q} \\[0.2cm]
   \hline \\[-0.3cm]
   \bar{A}^{\lambda}_{p,-p}  & A^{\lambda}_{q,t-q}
\end{array}\right).
\]
Since
\[
n^2+\frac{n_p(\lambda)(n_p(\lambda)-1)}{2}+n_p(\lambda)n_q(\lambda)+\frac{n_q(\lambda)(n_q(\lambda)-1)}{2}
=n^2+2n^2-n = n(n-1)
\] 
is even, the sign cancels, and $p$ and $q$ can be interchanged.
\end{proof}

The remaining results in this section deal with determinants of block matrices, which will prove useful in evaluating the other classical characters.
We note that we have not found our identities in Krattenthaler's treatises~\cite{kratt-1999,kratt-2005}.

\begin{lem}
\label{lem:matrix}
For $i = 1,\dots, k$, let $T_i$ be matrices of order $\ell_i \times m_i$ such that $\ell_1 + \cdots + \ell_k = m_1 + \cdots + m_k = d$. Define block-diagonal and block-antidiagonal matrices
\[
U \coloneqq \left(
\begin{array}{ccccc}
T_1 &&& \\
 & T_2  & & \text{\huge0}\\
 && \ddots \\
  \text{\huge0} &    &   & T_k
\end{array}
\right)
\quad
\text{and}
\quad
V \coloneqq \left(
\begin{array}{ccccc}
 &&& T_1 \\
\text{\huge0} &   & T_2 & \\
 & \iddots &&\\
 T_k  &    &   & \text{\huge0}
\end{array}
\right).
\]
Then
\[
\det(U) = (-1)^{ \sum_{1 \leq i < j \leq k } m_im_j} \det(V) = 
\begin{cases}
 0 & \text{if } \ell_i \neq m_i \text{ for some } i,\\
\ds \prod_{i=1}^k \det(T_i) & \text{otherwise}.
\end{cases}
\]
\end{lem}

\begin{proof}
It is easy to see that if $\ell_i = m_i$ for all $i$, then 
\[
\det(U) = \ds \prod_{i=1}^k \det(T_i), \quad \det(V)=(-1)^{\sum_{1 \leq i < j \leq k } m_im_j } \det(U).
\]

Now, assume $T_{i}$ is not a square matrix for some $i \in [k]$. 
Suppose first that $\ell_{i} < m_{i}$. Then, since rank($T_{i}^{T}T_{i})\leq$ rank($T_i) \leq \ell_i<$order($T_{i}^{T}T_{i}$), $\det (T_{i}^{T}T_{i})=0$. Therefore,
    \[
    (\det U)^2=\det U^{T}U =\prod_{j=1}^{k} \det (T_j^{T}T_j)=0,
\]
which implies $\det(U)=0$, and thus $\det(V)=0$.
If $\ell_{i} > m_{i}$, a similar calculation using the rank of $T_{i}T_{i}^{T}$ yields the same result.
\end{proof}

\begin{lem} 
\label{lem:det-blockmatrix}
Suppose $u_1,\dots,u_k$ are positive integers summing up to $kn$.
Further, let $\left(\gamma_{i,j}\right)_{1 \leq i \leq k, 1 \leq j \leq k+1}$ 
be a matrix of parameters such that $\gamma_{i,k+1}=\gamma_{i,k}$,  $1 \leq i \leq k$ and $\Gamma$ be the square matrix consisting of its first $k$ columns.
Let $U_{j}$ and $V_{j}$ be matrices of order $n \times u_j$ for $j \in [k]$.
Finally, define a $kn \times kn$ matrix 
with $k \times k$ blocks as
\[
\ds \Pi
\coloneqq \left(
\begin{array}{c|c}
 \left(\gamma_{i,2j-1}U_{j}-\gamma_{i,2j}V_{j}\right)_{\substack{1 \leq i\leq k\\1 \leq j\leq \floor{\frac{k+1}{2}}}}
 & 
 \left(\gamma_{i,2k+2-2j}U_{j}-\gamma_{i,2k+1-2j}V_{j}\right)_{\substack{1 \leq i\leq k \\\floor{\frac{k+3}{2}} \leq j\leq k}}
 \end{array}
 \right).
\] 

\begin{enumerate}
\item If $u_p+u_{k+1-p} \neq 2n$  for some $p \in [k]$, then $\det \Pi = 0$.  

\item If $u_p+u_{k+1-p}=2n$ for all $p \in [k]$, then
 \begin{equation}
   \label{lem2}
 \det \Pi
 = (-1)^{\Sigma} 
 (\det \Gamma)^n
 \prod_{i=1}^{\floor{\frac{k+1}{2}}}
 \det W_i, 
 \end{equation}
where
\[
W_i= \begin{cases}
\left(\begin{array}{c|c}U_i  & -V_{k+1-i} \\\hline
-V_i  & U_{k+1-i}
\end{array}\right) & 1 \leq i \leq \floor{\frac{k}{2}}, \\
\left( U_{\frac{k+1}{2}}-V_{\frac{k+1}{2}}  \right) 
& \text{$k$ odd and $i = \frac{k+1}{2}$},
\end{cases}
\]
and
\[
\Sigma=\sum_{i=1}^{\floor{\frac{k}{2}}}(n+u_{i})+\begin{cases}
0 &  k \text{ even},\\
n \ds \sum_{i=1}^{\frac{k-1}{2}} u_i &  k \text{ odd}.
\end{cases}
\]
\end{enumerate}

\end{lem}

\begin{proof}
Consider the permutation $\zeta$ in $S_{kn}$ which rearranges the columns of $\Pi$ blockwise in the following order: $1, k, 2, k-1, \dots$. 
In other words, $\zeta$ can be written in one line notation as
\begin{multline*}  
  \zeta = (\underbrace{1,\dots,u_1}_{u_1},\underbrace{kn-u_k+1,\dots,kn}_{u_k}, \\
  \underbrace{u_1+1,\dots,u_1+u_2}_{u_2},\underbrace{kn-u_k-u_{k-1}+1,\dots,kn-u_k}_{u_{k-1}},\dots).
  \end{multline*}
Then, the number of inversions of $\zeta$ is
\begin{equation}
    \begin{split}
        \label{inversion}
 \inv(\zeta)
 = \sum_{i=\floor{\frac{k+3}{2}}}^k u_i(kn-&(u_1+\dots+u_{k+1-i})-(u_i+\dots+u_k)) \\
= &\sum_{i=1}^{\floor{\frac{k}{2}}} u_{k+1-i}(kn-(u_1+\dots+u_{k+1-i})-(u_i+\dots+u_k)) .
    \end{split}
\end{equation}
 Then it can be seen that
 \begin{equation}
\label{mk}
  \det \Pi =
  \sgn (\zeta)
  \det\left(\begin{array}{c}
 \gamma_{i,j}U_{j''}-\gamma_{i,j'}V_{j''}
 \end{array}\right)_{1 \leq i,j\leq k},   
\end{equation}
 where 
\[
j'= j-(-1)^j \quad \text{and} \quad 
j''=\begin{cases}
 \frac{j+1}{2} & j \text{ odd}, \\
 k+1-\frac{j}{2} & j \text{ even}.
 \end{cases}
\] 
Now note that
\[
\left(\begin{array}{c}
 \gamma_{i,j}U_{j''}-\gamma_{i,j'}V_{j''}
 \end{array}\right)_{1 \leq i,j\leq k} 
 = \begin{pmatrix}
\gamma_{i,j} I_n
\end{pmatrix}_{1 \leq i,j\leq k}   \times
\left(
\begin{array}{ccccc}
W_1 &&& \\
 & W_2  & &\text{\huge0} \\
 &     & \ddots \\
\text{\huge0} &    &   & W_{\floor{\frac{k+1}{2}}}
\end{array}
\right).\]

Now, the matrix $(\gamma_{i,j} I_n)_{1 \leq i,j \leq k}$ can be written as a tensor product $\Gamma \otimes I_n$ and therefore 
$\det(\gamma_{i,j} I_n)_{1 \leq i,j \leq k} = \left(\det \Gamma \right)^n$.
Therefore, 
\begin{equation}
 \label{mk2}
  {\det \left(\begin{array}{c}
 \gamma_{i,j}U_{j''}-\gamma_{i,j'}V_{j''}
 \end{array}\right)_{1 \leq i,j\leq k}} =\left(\det \Gamma \right)^n 
 \det \left(
\begin{array}{ccccc}
W_1 &&& \\
 & W_2  & &\text{\huge0} \\
 &     & \ddots \\
\text{\huge0} &    &   & W_{\floor{\frac{k+1}{2}}}
\end{array}
\right).  
\end{equation}
 If $u_p+u_{k+1-p} \neq 2n$, 
 for some $p \in [\floor{\frac{k+1}{2}}]$, 
 then
     $W_p$ is not a square matrix.
     Using \cref{lem:matrix}, {we see that} the latter determinant is zero,
Hence, by \eqref{mk2} and \eqref{mk},
     \[
\det \Pi = 0.
\]
Now suppose $u_p+u_{k+1-p}=2n, 
 \forall$ $p \in [\floor{\frac{k+1}{2}}]$. Then 
     $W_p$ is a square matrix
     $\forall$ $p \in [\floor{\frac{k+1}{2}}]$. 
     So, by  \eqref{mk2}, we get,
\begin{equation*}
 {\det \left(\begin{array}{c}
 \gamma_{i,j}U_{j''}-\gamma_{i,j'}V_{j''}
 \end{array}\right)_{1 \leq i,j\leq k}}
 =\left( \det \Gamma \right)^n
 \prod_{i=1}^{\floor{\frac{k+1}{2}}} \det W_i.
\end{equation*}
All that remains is to compute the sign.
By \eqref{inversion}, we get
\[ \inv(\zeta) 
 = \sum_{i=1}^{\floor{\frac{k}{2}}} (2n-u_i)(kn-2(k-i-1)).\] 
Therefore, if $k$ is even, then $\inv(\zeta)$ is even and $\sgn(\zeta)$ is $1$. If $k$ is odd, then the only contribution for $\sgn(\zeta)$ comes from $n \ds \sum_{i=1}^{\frac{k-1}{2}} u_i$, since other terms are even. 
Summing the terms gives $\Sigma$, completing the proof.
\end{proof}

\section{Schur factorization}
\label{sec:schur}

We first give a self-contained proof of the result of Littlewood~\cite{littlewood-1950} and Prasad~\cite{prasad-2016}, \cref{thm:schur-fac}. {We later found out that Littlewood's strategy of proof is, although in different language, essentially the same as ours.}

\begin{proof}[Proof of \cref{thm:schur-fac}] 
Recall that $\lambda$ has length at most $tn$.
From the definition \eqref{gldef}, 
the desired Schur polynomial is
\begin{equation}
  \label{srs}
   s_{\lambda}(X,\omega X, \dots ,\omega^{t-1}X)
   = \frac{\det
   \left( \left( (\omega^{p-1}x_i)^{\beta_{j}(\lambda)}
   \right)_{\substack{1 \leq i\leq n \\ 1\leq j\leq tn}} 
   \right)_{1\leq p \leq t}}
   {\det\left( \left( (\omega^{p-1}x_i)^{tn-j}
   \right)_{\substack{1 \leq i\leq n \\ 1\leq j\leq tn}} 
   \right)_{1\leq p \leq t}}.
\end{equation}
Permuting the columns of the determinant in the numerator of \eqref{srs} by $\sigma_{\lambda}$ from \eqref{sigma-perm},{ we see that the numerator of \eqref{srs} is}
\begin{equation}
    \begin{split}
       \label{sdet}
     \sgn(\sigma_{\lambda}) &
     \det\left( \left( (\omega^{p-1}x_i)^{\beta_{\sigma_{\lambda}(j)}(\lambda)}
     \right)_{\substack{1 \leq i\leq n \\ 1\leq j\leq tn}} 
     \right)_{1\leq p \leq t} \\
    = & \ds\sgn(\sigma_{\lambda}) 
    \det  \left( \omega^{(p-1)(q-1)}
    \left( x_i^{\beta_{j}^{(q-1)}(\lambda)}
    \right)_{\substack{1 \leq i\leq n \\  1 \leq j \leq  n_{q-1}(\lambda)}} \right)_{1\leq p,q\leq t} \\
     & \qquad =  \ds \sgn(\sigma_{\lambda}) 
     \det\left(  \omega^{(p-1)(q-1)}A^{\lambda}_{q-1} \right)_{1\leq p,q\leq t}, 
    \end{split}
\end{equation}
where $A^{\lambda}_{q-1}$ is defined in \eqref{def AB}.
Note that 
\[
\ds \left(
\omega^{(p-1)(q-1)}A^{\lambda}_{q-1}
\right)_{1 \leq p,q \leq t} 
= \left(
\omega^{(p-1)(q-1)} \otimes I_n
\right)_{1 \leq p,q \leq t} 
\times \left(
\begin{array}{ccccc}
A^{\lambda}_{0} &&& \\
 & A^{\lambda}_{1}  & &\text{\huge0} \\
\text{\huge0} &     & \ddots \\
&    &   & A^{\lambda}_{t-1}
\end{array}
\right),
\]
where $I_n$ is the $n \times n$ identity matrix. 
Hence,
\begin{equation}
\begin{split}
 \label{sfinal}
 \ds \det 
  & \left(
 \omega^{(p-1)(q-1)}A^{\lambda}_{q-1}
 \right)_{1 \leq p,q \leq t}  \\
 = & \left(
 \det \left(\omega^{(p-1)(q-1)}
 \right)_{1 \leq p,q \leq t}
 \right)^n 
 \times \det \left(
\begin{array}{ccccc}
A^{\lambda}_{0} &&& \\
 & A^{\lambda}_{1}  & &\text{\huge0} \\
\text{\huge0} &     & \ddots \\
&    &   & A^{\lambda}_{t-1}
\end{array}
\right).   
\end{split}
 \end{equation}
If core$_t(\lambda)$ is not {empty}, then using \cref{cor:ary}, {we see that} $n_{q}(\lambda)\neq n$ for some $0 \leq q \leq t-1$. So, $A^{\lambda}_{q}$ is not a square matrix for some $0 \leq q \leq t-1$. By \cref{lem:matrix},
$ \ds \det \left(\omega^{(p-1)(q-1)}A^{\lambda}_{q-1}\right)_{1 \leq p,q \leq t}=0$ and hence
\begin{equation*}
    s_{\lambda}(X,\omega X, \dots ,\omega^{t-1}X) = 0.
\end{equation*}

If $\core \lambda t$ is {empty}, then \cref{cor:ary} {shows} that
$n_{q}(\lambda)= n$ for all $0 \leq q \leq t-1$ and $A^{\lambda}_{q}$ is a square matrix for all $0 \leq q \leq t-1$. 
Applying \cref{lem:matrix} again to \eqref{sfinal}, {we see that}
\begin{equation*}
 \det \left(\omega^{(p-1)(q-1)}A^{\lambda}_{q-1}\right)_{1 \leq p,q \leq t}
 =\left( \det \left(
 \omega^{(p-1)(q-1)} \right)_{1 \leq p,q \leq t}\right)^n
 \prod_{q=0}^{t-1} 
 \det A^{\lambda}_{q}.   
\end{equation*}
Substituting in \eqref{sdet}, {we see that}
the numerator of \eqref{srs} is
\begin{equation}
    \label{sdet1}
   \sgn(\sigma_{\lambda}) 
   \left( \det \left(
 \omega^{(p-1)(q-1)}
 \right)_{1 \leq p,q \leq t}
 \right)^n 
 \prod_{q=0}^{t-1} 
 \det A^{\lambda}_{q}.   
\end{equation}
Evaluating \eqref{sdet1} for the empty partition and using \eqref{sigma0}, {we see that} the denominator of \eqref{srs} is
 \begin{equation}
\label{sfinal1}
(-1)^{\frac{t(t-1)}{2}\frac{n(n+1)}{2}} \left( \det \left(\omega^{(p-1)(q-1)} \right)_{1 \leq p,q \leq t}\right)^n \prod_{q=0}^{t-1} 
 \det A_{q}.   
 \end{equation}
{Substitution of the values \eqref{sdet1} and \eqref{sfinal1}} in \eqref{srs} gives
\begin{equation}
\label{sdet2}
    s_{\lambda}(X,\omega X, \dots ,\omega^{t-1}X)
    =(-1)^{\frac{t(t-1)}{2}\frac{n(n+1)}{2}}
    \sgn(\sigma_{\lambda}) \prod_{q=0}^{t-1} 
    \frac{\det A^{\lambda}_{q}}{\det A_{q}},
\end{equation}
where $A_{q}$ is defined in \eqref{def AB0}. 
Hence, using \eqref{snew} in \eqref{sdet2} gives
\begin{equation*}
  s_{\lambda}(X,\omega X, \dots ,\omega^{t-1}X) 
  =  (-1)^{\frac{t(t-1)}{2}\frac{n(n+1)}{2}}
  \sgn(\sigma_{\lambda})
  \prod_{i=0}^{t-1} s_{\lambda^{(i)}}(X^t), 
  \end{equation*}
completing the proof.  
\end{proof}

We will now prove \cref{thm:schur-1}. 
From \cref{prop:mcd-t-core-quo}, it follows that knowing $n_{i}(\lambda,m)$, $0 \leq i \leq t-1$, for a $t$-core $\lambda$ of length {at most} $m$ is the same as knowing $\lambda$. 
For the rest of this section, we will assume that $\ell(\lambda) \leq tn+1$. This implies $\ell(\core \lambda t) \leq tn+1$.
We then have the following result.  

\begin{lem}
\label{lem:np-schur-1}
Suppose $\lambda$ is a partition 
and $c \in \mathbb{N}_{\geq 0}$. Then
\[ 
\core \lambda t = (c) \text{ {if and only if} }  n_{i}(\lambda)
=
\begin{cases}
n+1 & \text{ if } i=c, \\
n & \text{ otherwise. }
\end{cases}
\]
\end{lem}

\begin{proof} 
It is obvious that $\core\lambda t=(c)$ if and only if 
$\beta(\core\lambda t)=(tn+c,tn-1,tn-2,\dots,0)$. This further implies
that $\core\lambda t=(c)$ if and only if 
\[
n_{i}(\core\lambda t)=\begin{cases}
n+1 & \text{ if } i=c,
\\
n & \text{ otherwise. }
\end{cases}.
\]
Now \eqref{no-parts-partition=core} gives the desired result.
\end{proof}

For $0 \leq c \leq t-1$, let $\sigma_{\lambda}^{c}$ be a permutation in $S_{tn+1}$ such that it rearranges parts of $\beta(\lambda)$ in the following way.
First $n_c(\lambda)$ parts of $\sigma_{\lambda}^{c}(\beta(\lambda))$ are congruent to $c \pmod{t}$, in descending order 
and then the next $n_i(\lambda)$ parts of $\sigma_{\lambda}^{c}(\beta(\lambda))$ are congruent to $i \pmod{t}$, in descending order, starting from $i=0$ to $t-1$, $i \neq c$.

{We now prove the second Schur factorization. Our proof strategy follows the proof sketch given in Littlewood~\cite[Chapter VII, Section IX]{littlewood-1950}.}

\begin{proof}[Proof of \cref{thm:schur-1}] 
Using the definition \eqref{gldef}, {we see that} the desired Schur polynomial is
\begin{equation}
 \label{n1srs}
   s_{\lambda}(X,\omega X, \dots ,\omega^{t-1}X,x) 
   = \frac{\det \left(
    \begin{array}{c}
    \left( \left( 
    (\omega^{p-1}x_i)^{\beta_{j}(\lambda)}
    \right)_{\substack{1 \leq i\leq n \\ 1\leq j\leq tn+1}}
    \right)_{1\leq p \leq t} \\[0.4cm]  
      \hline     \\[-0.3cm]
       \left( x^{\beta_{j}(\lambda)} \right)_{1 \leq j \leq tn+1}
    \end{array} 
    \right)}
    {\det \left( \begin{array}{c}  
    \left( \left( (\omega^{p-1}x_i)^{tn+1-j}
    \right)_{\substack{1 \leq i\leq n \\ 1\leq j\leq tn+1}}
    \right)_{1\leq p \leq t} \\[0.4cm]  
      \hline     \\[-0.3cm]
    \left( x^{\beta_{j}(\lambda)} \right)_{1 \leq j \leq tn+1} \end{array} \right)}.
\end{equation}
By the pigeonhole principle, there exists a $v$, $0 \leq v \leq t-1$, such that $n_{v}(\lambda) \geq n+1$.
Permuting the columns of the determinant in the numerator of \eqref{n1srs} by $\sigma^{v}_{\lambda}$, {we see that} the numerator of \eqref{n1srs} is
\begin{equation}
    \begin{split}
    \sgn(\sigma^{v}_{\lambda}) \det & \left(
    \begin{array}{c}
        \left( \left( (\omega^{p-1}x_i)^{\beta_{\sigma^{v}_{\lambda}(j)}(\lambda)}
        \right)_{\substack{1 \leq i\leq n \\ 1\leq j\leq tn+1}}
        \right)_{1\leq p \leq t} \\[0.4cm]  
      \hline     \\[-0.3cm]
    \left( x^{\beta_{\sigma^{v}_{\lambda}(j)}(\lambda)} \right)_{1 \leq j \leq tn+1}
    \end{array} \right) \\
    & = \sgn(\sigma^v_{\lambda}) \det \left(
    \begin{array}{c|c} 
    \left( \omega^{(p-1)(v)} A^{\lambda}_{v} \right)_{1\leq p \leq t} &   \left( \omega^{(p-1)(q-1)} A^{\lambda}_{q-1} 
        \right)_{\substack{1\leq p,q \leq t\\q \neq v+1}} \\[0.4cm]  
      \hline     \\[-0.3cm]
  B^{\lambda}_{v}  & \left( B^{\lambda}_{q-1} \right)_{\substack{1 \leq q \leq t \\q \neq v+1}}
    \end{array} \right),
    \end{split}
\end{equation}
where $A^{\lambda}_{q}$ is defined in \eqref{def AB} and  $B^{\lambda}_{q}=\left( x^{\beta^{(q)}_j(\lambda)} \right)_{1 \leq j \leq n_{q}(\lambda)}, q \in [0,t-1]$. For $p \in [t],$ multiplying the rows of the $p$'th block by $\omega^{(1-p)v},$ {we see that }the numerator of \eqref{n1srs} is 
\[
\sgn(\sigma^v_{\lambda}) \omega^{nv\frac{t(t-1)}{2}} \det \left(
    \begin{array}{c|c} 
    \left(  A^{\lambda}_{v} \right)_{1\leq p \leq t} &   \left( \omega^{(p-1)(q-v-1)} A^{\lambda}_{q-1} 
        \right)_{\substack{1\leq p,q \leq t\\q \neq v+1}} \\[0.4cm]  
      \hline     \\[-0.3cm]
  B^{\lambda}_{v}  & \left( B^{\lambda}_{q-1} \right)_{\substack{1 \leq q \leq t \\q \neq v+1}}
    \end{array} \right).
\]
Applying the blockwise row operations 
$R_t \rightarrow R_1+R_2+\dots+R_{t}$ 
followed by $R_i \rightarrow R_i-\frac{1}{t}R_t$, for $1 \leq i \leq t-1$, we get 
\begin{equation}
\label{3nsrs}
    \sgn(\sigma^v_{\lambda})  \omega^{nv\frac{t(t-1)}{2}}
\det  \left(\begin{array}{c|c}
\left(  0 \right)_{1\leq p \leq t-1}   & \left( \omega^{(p-1)(q-v-1)} A^{\lambda}_{q-1} 
        \right)_{\substack{1\leq p \leq t-1\\1\leq q \leq t\\q \neq v+1}}   
\\&\\\hline&\\
t A^{\lambda}_{v} &   \left(  0 \right)_{\substack{1 \leq q \leq t \\q \neq v+1}}
\\&\\\hline&\\
B^{\lambda}_{v}  & \left( B^{\lambda}_{q-1} \right)_{\substack{1 \leq q \leq t \\q \neq v+1}}  \end{array}\right).
\end{equation}
Note that
\begin{equation}
\begin{split}
\label{7nsrs}
\ds & \left( \omega^{(p-1)(q-v-1)} A^{\lambda}_{q-1} 
        \right)_{\substack{1\leq p \leq t-1\\1\leq q \leq t\\q \neq v+1}}  \\
& = \left(
\omega^{(p-1)(q-v-1)} \otimes I_n
\right)_{\substack{1\leq p \leq t-1\\1\leq q \leq t\\q \neq v+1}}  
\times \left(
\begin{array}{ccccc}
A^{\lambda}_0  \\
& \ddots  &   & \text{\huge0}\\
& & \widehat{A^{\lambda}_{v}} \\
& \text{\huge0} &   & \ddots \\
&  &   &   & A^{\lambda}_{t-1} 
\end{array}
\right). 
\end{split}
\end{equation}

Now suppose $\ell(\core\lambda t) > 1$. 
If $n_v(\lambda)>n+1$, then $\left( 
\begin{array}{c}
    A^{\lambda}_{v}  \\ \hline 
   B^{\lambda}_{v}   
\end{array}
\right)$ is not a square matrix. 
Therefore, the matrix in \eqref{3nsrs} is of the form
\[
\left( 
\begin{array}{c|c}
    0 & M_1  \\ \hline 
   M_2 & M_3 
\end{array}
\right), 
\]
where the number of columns in $M_2$ is more than its number of rows.
Thus, the determinant in \eqref{3nsrs} is $0$. 
If $n_{v}(\lambda) = n+1$, then $\left( 
\begin{array}{c}
    A^{\lambda}_{v}  \\ \hline 
   B^{\lambda}_{v}   
\end{array} 
\right)$ is a square matrix. 
So the numerator of \eqref{n1srs}, using \eqref{3nsrs} and \eqref{7nsrs}, is 
\begin{equation}
    \label{4nsrs}
    \begin{split}
    \sgn(\sigma^v_{\lambda}) & \omega^{nv\frac{t(t-1)}{2}}   (-1)^{(t-1)n(n+1)} t^n \det \left( 
\begin{array}{c}
    A^{\lambda}_{v}  \\ \hline 
   B^{\lambda}_{v}   
\end{array}
\right) \\ & \times \det \left(
\omega^{(p-1)(q-v-1)} \otimes I_n
\right)_{\substack{1\leq p \leq t-1\\1\leq q \leq t\\q \neq v+1}}  
\times \det  \left(
\begin{array}{ccccc}
A^{\lambda}_0  \\
& \ddots  &   & \text{\huge0}\\
& & \widehat{A^{\lambda}_{v}} \\
& \text{\huge0} &   & \ddots \\
&  &   &   & A^{\lambda}_{t-1} 
\end{array}
\right).
 \end{split}
\end{equation}
Since $\core \lambda t \neq (v)$, $n_{i}(\lambda) \neq n$ for some $i \in \{0,1,\dots,\widehat{v},\dots,t-1\}$ by \cref{lem:np-schur-1}. So,
$A^{\lambda}_{i}$ is not a square matrix. Using \cref{lem:matrix}, {we see that}
the last determinant in \eqref{4nsrs} is $0$.
Plugging this into \eqref{n1srs}, we have that
\begin{equation*}
    s_{\lambda}(X,\omega X, \dots ,\omega^{t-1}X,x) = 0.
\end{equation*}
If $\ell(\core\lambda t) \leq 1$, then $\core\lambda t=(c)$ for some $0 \leq c \leq t-1$.
Using \cref{lem:np-schur-1}, {we see that}
\[
n_{i}(\lambda)=\begin{cases}
n+1 & \text{ if } i=c,
\\
n & \text{ otherwise.}
\end{cases}\]
Since $n_{v}(\lambda) \geq n+1$, we must have that $v=c$ and $n_{v}(\lambda) = n+1$. Thus, $A^{\lambda}_{i}$ is {a} square matrix for all $0 \leq i \leq t-1, i \neq v$. Using \eqref{3nsrs} and \eqref{7nsrs}, {we see that} the numerator of \eqref{n1srs} is
\begin{equation}
\begin{split}
    \label{5nsrs}
    \sgn(\sigma^c_{\lambda})   \omega^{cn\frac{t(t-1)}{2}}   (-1)^{(t-1)n(n+1)} t^n & \det \left( 
\begin{array}{c}
    A^{\lambda}_{c}  \\ \hline 
   B^{\lambda}_{c}   
\end{array}
\right) \\ \times
& \left( \det \left( \omega^{(p-1)(q-c-1)} 
        \right)_{\substack{1\leq p \leq t-1\\1\leq q \leq t\\q \neq c+1}} \right)^n \prod_{\substack{i=0\\i \neq c}}^{t-1} \det  A^{\lambda}_{i}.
        \end{split}
\end{equation}
Since the denominator of \eqref{n1srs} is its numerator evaluated at the empty partition, the above strategy also works for the denominator. For the empty partition,  $n_{0}(\emptyset,tn+1)=n+1$, $n_{q}(\emptyset,tn+1)=n$, $1 \leq q \leq t-1$ 
with $\sgn(\sigma^0_{\emptyset})
=(-1)^{\frac{t(t-1)}{2}\frac{n(n+1)}{2}}$.  This is then given by
\begin{equation}
\begin{split}
    \label{nfac3}
   (-1)^{\frac{t(t-1)}{2}\frac{n(n+1)}{2}} (-1)^{(t-1)n(n+1)} t^n & \det \left( 
\begin{array}{c}
    A_{0}  \\ \hline 
   B_{0}   
\end{array}
\right) \\ {\times} &
\left( \det \left( \omega^{(p-1)(q-1)} 
        \right)_{\substack{1\leq p \leq t-1\\1\leq q \leq t\\q \neq c+1}} \right)^n \prod_{i=1}^{t-1} \det A_{i}.
        \end{split}
        \end{equation}
Taking ratios and using \eqref{snew}, {we see that} the Schur polynomial is given by
\begin{equation*}
    (-1)^{\frac{t(t-1)}{2}\frac{n(n+1)}{2}-cn}
  \sgn(\sigma^{c}_{\lambda}) \times  x^c \times
  s_{\lambda^{(c)}}(X^t,x^t) 
  \prod_{\substack{i=0\\i\neq c}}^{t-1} 
  s_{\lambda^{(i)}}(X^t),
\end{equation*}
which proves the result.
\end{proof}

There is a natural correspondence between irreducible representations of $\Sp_{2tn}$ and $\GL_{2tn+1}$ which is one-one
and onto the set of irreducible selfdual representations of $\GL_{2tn+1}$; see for instance~\cite{prasad-wagh-2020}.
The following lemma then relates the character values on these special elements in the two groups; at least it says that being nonzero is exactly the same condition as in correspondence between irreducible representations of $\Sp_{2n}$ and $\GL_{2n+1}$ in~\cite{prasad-wagh-2020}.
To be precise, let $\mu$ be a partition of length at most $tn$ indexing a representation of $\Sp_{2tn}$, and construct
$\lambda=\mu_1+(\mu,0,\dots,0,-\rev{\mu})$, which has length at most $2tn+1$.
Then $\lambda$ indexes a self-dual representation of $\GL_{2tn+1}$.

\begin{lem}
With $\mu$ and $\lambda$ as above, $\ell{(\core \lambda t)}$ is at most one if and only if $\core \mu t$ is symplectic.
\end{lem}

\begin{proof}
Let $\mu_1 \equiv p \pmod{t}$ for some $p \in [0,t-1]$. Observe that
\begin{multline*}
\beta(\lambda)=(2\mu_1+2tn,\mu_1+\mu_2+2tn-1,\dots,\mu_1+\mu_{tn}+tn+1,
\mu_1+tn,\\
\mu_1-\mu_{tn}+tn-1,\dots,\mu_1-\mu_2+1,0),     
\end{multline*}
where the first $tn$ parts are of the form $\mu_1+tn+1+\beta(\mu)$ and the last $tn$ parts are of the form $\mu_1+tn-1-\rev(\beta(\mu))$. 
The sum of $\beta(\lambda)_i$ and $\beta(\lambda)_{2tn+2-i}$ is congruent to $2p \pmod{t}, i \in [tn]$. So for all $j \in [0,t-1]$, the parts of $\beta(\lambda)$ less than $\mu_1+tn$ congruent to $j$ are in bijective correspondence with the parts greater than $\mu_1+tn$ congruent to $2p-j$.
Hence the number of parts of $\beta(\lambda)$ greater than $\mu_1+tn$ congruent to $j$ and $2p-j$ is 
\begin{equation}
 \label{nmu}
\begin{cases}
n_j(\lambda)-1 & \text{ if } j = p, \\
n_j(\lambda) & \text{ otherwise },
\end{cases}   
\end{equation}
and $n_p(\lambda)$ is odd. If $\zeta \equiv i \pmod{t}$ occurs in $\beta(\mu),$ then $\mu_1+tn+1+\zeta \equiv p+1+i \pmod{t}$, greater than $\mu_1+tn$, occurs in $\beta(\lambda)$. Using \eqref{nmu} for $j=i_p \coloneqq (p+1+i) \pmod{t},$ {we see that} 
\begin{equation}
\label{nmu1}
 n_i(\mu)+n_{t-2-i}(\mu)= n_{i_p}(\lambda), i \in [0,t-2], \quad n_{t-1}(\mu)= \frac{n_{p}(\lambda)-1}{2}.  
\end{equation}
Suppose $\ell{(\core \lambda t)}$ is at most one. Then $\core \lambda t=(c)$, for some $c \in [0,t-1]$. \cref{lem:np-schur-1} shows that $n_c(\lambda)=2n+1$ and $n_i(\lambda)=2n$ for $i \in [0,t-1], i \neq c$. Since $n_p(\lambda)$ is odd, $p=c$. {Substitution of this} in \eqref{nmu1} gives, $n_i(\mu)+n_{t-2-i}(\mu)=2n$ for $i \in [0,t-2]$ and $n_{t-1}(\mu)=n$. 
From \cref{cor:symp-parts}, it now follows that $\core \mu t$ is symplectic.

Now assume $\core \mu t$ is symplectic. Using \cref{cor:symp-parts} in \eqref{nmu1}, {we see that} $n_p(\lambda)=2n+1$ and $n_i(\lambda)=2n$. Then \cref{lem:np-schur-1} shows that $\core \lambda t=(p)$.
\end{proof}

\section{Factorization of other classical characters}
\label{sec:fact other}

In this section, we will prove all the other classical character factorizations using results from \cref{sec:background}. We will give the most details for the symplectic case in \cref{sec:symp} and will be a little more sketchy for the even orthogonal case in \cref{sec:eorth} and the odd orthogonal case in \cref{sec:oorth}.
We will assume $\ell(\lambda) \leq tn$ throughout this section.

\subsection{Symplectic characters}
\label{sec:symp}

We first recall the matrices 
$A_{p,q}^{\lambda}$ and $\bar{A}_{p,q}^{\lambda}$ from \eqref{def AB}.
If $\ds \sum_{i=0}^{t-2}n_{i}(\lambda)=(t-1)n$, then consider the $(t-1)n \times (t-1)n$ matrix
\[
\Pi_1= \left(\omega^{pq}A^{\lambda}_{q-1,1}-\bar{\omega}^{pq}
\bar{A}^{\lambda}_{q-1,1}\right)_{1 \leq p,q \leq t-1}.
\]
{Substitution of} $U_j=A_{j-1,1}^{\lambda}$, $V_j=\bar{A}_{j-1,1}^{\lambda}$ for $1 \leq j \leq t-1$ and
\[
\gamma_{i,j} = \begin{cases}
\omega^{\frac{i(j+1)}{2}} & j \text{ odd},\\
\omega^{-\frac{ij}{2}} & j \text{ even}, 
\end{cases}
\]
in \cref{lem:det-blockmatrix} proves the following corollary.

\begin{cor} 
\label{cor:det-Pi_1} 
\begin{enumerate} 

\item If $n_{i}(\lambda)+n_{t-2-i}(\lambda)\neq 2n$ for some $i \in [0,\floor{\frac{t-2}{2}}]$, then $\det \Pi_1 = 0$.

\item If $n_{i}(\lambda)+n_{t-2-i}(\lambda) = 2n$ for all $i \in \{0,1,\dots,\floor{\frac{t-2}{2}}\}$, then 
\begin{equation}
    \begin{split}
        \label{spdet3}
\det \Pi_1 
= (-1)^{\Sigma_1} \left(\det (\gamma_{i,j})_{1 \leq i,j \leq t-1} \right)^{n}
\prod_{q=1}^{\floor{\frac{t-1}{2}}}
&\det\left(\begin{array}{c|c}
   A_{q-1,1}  & \bar{A}_{t-q-1,1} \\[0.2cm]
   \hline \\[-0.3cm]
   \bar{A}_{q-1,1}  & A_{t-q-1,1}
\end{array}\right) \\
\times &
\begin{cases}
\det \left(A_{\frac{t}{2}-1,1}- \bar{A}_{\frac{t}{2}-1,1} \right)  
& t \text{ even,}\\
1 & t \text{ odd,}
\end{cases}  
    \end{split}
\end{equation}
where 
\[
\Sigma_1=
\sum_{q=1}^{\floor{\frac{t-1}{2}}}
\left(n+n_{q-1}(\lambda)\right)
+\begin{cases}
n \ds \sum_{q=1}^{\frac{t-2}{2}} n_{q-1}(\lambda) &  t \text{ even}, \\ 
0 & t \text{ odd}.
\end{cases}
\]
\end{enumerate}
\end{cor}

\begin{proof}[Proof of \cref{thm:sympfact}] 
Using the formula {for} symplectic characters in \eqref{spdef}, {we see that} the symplectic polynomial considered here is
\begin{equation}
 \label{symp-1}
    \sp_{\lambda}(X,\omega X,\dots ,\omega^{t-1}X) 
    = \frac{
    \det
    \left( \left(
    (\omega^{p}x_i)^{\beta_{j}(\lambda)+1}
    -(\bar{\omega}^{p}\bar{x}_i)^{\beta_{j}(\lambda)+1}
    \right)_{\substack{1 \leq i\leq n \\ 1\leq j\leq tn}} 
    \right)_{0 \leq p \leq t-1}}
    {\det \left( \left( (\omega^{p}x_i)^{tn-j+1}
    -(\bar{\omega}^{p}{\x}_i)^{tn-j+1}
    \right)_{\substack{1 \leq i\leq n \\ 1\leq j\leq tn}}
    \right)_{0 \leq p \leq t-1}}.   
\end{equation}
Since the denominator of the right hand side of \eqref{symp-1} is the same as its numerator evaluated at the empty partition, we compute the factorization for the numerator and use that to get factorization for the denominator. Permuting the columns of the determinant in the numerator of \eqref{symp-1} by $\sigma_{\lambda}$ from \eqref{sigma-perm}, {we see that} the numerator of \eqref{symp-1} is 
\begin{equation}
    \begin{split}
        \label{spfinal}
\sgn(\sigma_{\lambda}) 
\det
\left(
\left( 
(\omega^{p}x_i)^{\beta_{\sigma_{\lambda}(j)}(\lambda)+1}
-(\bar{\omega}^{p}\bar{x}_i)^{\beta_{\sigma_{\lambda}(j)}(\lambda)+1}
\right)_{\substack{1 \leq i\leq n \\ 1\leq j\leq tn}} \right)_{0 \leq p \leq t-1} \\
    =\sgn(\sigma_{\lambda}) 
    \det
    \left( 
    \left(
    \omega^{p(q+1)}x_i^{\beta_j^{(q)}(\lambda)+1}
    -\bar{\omega}^{p(q+1)}\bar{x}_i^{\beta_j^{(q)}(\lambda)+1}
    \right)
    _{\substack{1 \leq i\leq n \\ 1 \leq j \leq  n_{q}(\lambda)}} 
    \right)_{0 \leq p,q\leq t-1}\\
   = \sgn(\sigma_{\lambda})
   \det \left( 
   \omega^{p(q+1)} A^{\lambda}_{q,1}
   -\bar{\omega}^{p(q+1)}\bar{A}^{\lambda}_{q,1}
   \right)_{0\leq p,q\leq t-1}. 
    \end{split}
\end{equation}
Applying the blockwise row operations 
$R_1 \rightarrow R_1+R_2+\dots+R_{t}$ 
followed by $R_i \rightarrow R_i-\frac{1}{t}R_1$, for $2 \leq i \leq t$, we get 
\begin{equation}
    \begin{split}
\label{det-1}
 \ds \det &\left( \omega^{p(q+1)} A^{\lambda}_{q,1}-\bar{\omega}^{p(q+1)}\bar{A}^{\lambda}_{q,1}\right)_{0\leq p,q\leq t-1}   \\
= &\det  \left(\begin{array}{c|c|c|c}
0 & \dots & 0 & 
t(A^{\lambda}_{t-1,1}-\bar{A}^{\lambda}_{t-1,1})
\\&&&\\\hline&&&\\
\omega A^{\lambda}_{0,1}-\omega^{t-1}\bar{A}^{\lambda}_{0,1} &  \dots & \omega^{t-1}A^{\lambda}_{t-2,1}-\omega \bar{A}^{\lambda}_{t-2,1}&0\\&&&\\\hline&&&\\
\omega^2 A^{\lambda}_{0,1}-\omega^{t-2}\bar{A}^{\lambda}_{0,1} &  \dots & \omega^{t-2}A^{\lambda}_{t-2,1}-\omega^2 \bar{A}^{\lambda}_{t-2,1}&0\\&&&\\\hline&&&\\
\vdots  & \ddots & \vdots & \vdots \\&&&\\\hline&&&\\
\omega^{t-1}A^{\lambda}_{0,1}-\omega \bar{A}^{\lambda}_{0,1}  & \dots & \omega A^{\lambda}_{t-2,1}-\omega^{t-1} \bar{A}^{\lambda}_{t-2,1}&0
\end{array}\right). 
    \end{split}
\end{equation}
This is now a $2 \times 2$ block determinant with anti-diagonal blocks. We apply \cref{lem:matrix}, for $k=2$ and $d=tn$, to evaluate this determinant.

If $\core \lambda t$ is not a symplectic $t$-core, then by \cref{cor:symp-parts},
either $n_{t-1}(\lambda) \neq n$ or  $n_{i}(\lambda)+n_{t-2-i}(\lambda) \neq  2n$ for some $i \in \left\{0,1,\dots,\floor{\frac{t-2}{2}}\right\}$.
In the first case, i.e. if $n_{t-1}(\lambda) \neq n$, then \cref{lem:matrix} shows
the determinant is \eqref{det-1} is $0$.
If $n_{t-1}(\lambda) = n$, then the determinant in \eqref{det-1} is
\begin{equation}
\label{spdet2}
(-1)^{(t-1)n^2} t^{n} 
\det \left(
A^{\lambda}_{t-1,1}-\bar{A}^{\lambda}_{t-1,1}
\right)
\times
\det \left( 
\omega^{pq} A^{\lambda}_{q-1,1}
-\bar{\omega}^{pq}\bar{A}^{\lambda}_{q-1,1}
\right)_{1\leq p,q\leq t-1},
\end{equation}
using \cref{lem:matrix}.
Observe that the $(t-1)n \times (t-1)n$ block  matrix of the determinant in \eqref{spdet2} is of the form $\Pi_1$ in \cref{cor:det-Pi_1}.
Now if $n_{i}(\lambda)+n_{t-2-i}(\lambda) \neq  2n$ for some $i \in \left\{0,1,\dots,\floor{\frac{t-2}{2}}\right\}$, then 
the determinant in \eqref{spdet2} is $0$ by \cref{lem:det-blockmatrix} 
and therefore, in both cases,
\[
\sp_{\lambda}(X,\omega X, \dots ,\omega^{t-1}X) = 0.
\]

If $\core \lambda t$ is a  symplectic $t$-core, then
by \cref{cor:symp-parts}, $n_{t-1}(\lambda) = n$ and $n_{i}(\lambda)+n_{t-2-i}(\lambda) =  2n$, $i \in \left\{0,1,\dots,\floor{\frac{t-2}{2}}\right\}$.
Using \cref{cor:det-Pi_1}(2) in \eqref{spdet2}, {we see that} the determinant in the numerator of \eqref{symp-1} is 
\begin{equation}
    \begin{split}
      \label{num}
\sgn(\sigma_{\lambda}) &
((-1)^{(t-1)n}t)^{n} 
\det \left(A^{\lambda}_{t-1,1}-\bar{A}^{\lambda}_{t-1,1}\right) 
(-1)^{\Sigma_1} 
(\det (\gamma_{i,j})_{1 \leq i,j \leq t-1} )^n \\&
\times \ds \prod_{q=1}^{\floor{\frac{t-1}{2}}}\det\left(\begin{array}{c|c}
   A^{\lambda}_{q-1,1}  & \bar{A}^{\lambda}_{t-q-1,1} \\[0.2cm]
   \hline \\[-0.3cm]
  \bar{A}^{\lambda}_{q-1,1}  & A^{\lambda}_{t-q-1,1}
\end{array}\right)  \times
\begin{cases}\det \left(A^{\lambda}_{\frac{t}{2}-1,1}- \bar{A}^{\lambda}_{\frac{t}{2}-1,1} \right)  & t \text{ even,}\\
1 & t \text{ odd}.
\end{cases}  
    \end{split}
\end{equation}
We now simplify the $2 \times 2$ block determinants.
For $1 \leq q \leq \floor{\frac{t-1}{2}}$, multiplying row $i$ in the top blocks of the matrix by $\bar{x}_i^{q}$ and row $i$ in the
bottom blocks by $x_i^{q}$ for each $i \in [n]$, we get
\begin{equation}
    \begin{split}
        \label{spdet5}
\det & \left(\begin{array}{c|c}
   A^{\lambda}_{q-1,1}  & \bar{A}^{\lambda}_{t-q-1,1} \\[0.2cm]
   \hline \\[-0.3cm]
   \bar{A}^{\lambda}_{q-1,1}  & A^{\lambda}_{t-q-1,1}
\end{array}\right)  \\ 
= \det & \left(\begin{array}{c|c}
   \left(
   x_i^{\beta_j^{(q-1)}(\lambda)+1-q}
   \right)_{\substack{1 \leq i\leq n \\ 1\leq j\leq n_{q-1}(\lambda)}}
   & \left (
   \bar{x}_i^{\beta_j^{(t-1-q)}(\lambda)+1+q}
   \right)_{\substack{1 \leq i\leq n \\ 1\leq j\leq n_{t-1-q}(\lambda)}}\\ \\
   \hline\\
   \left(
   \bar{x}_i^{\beta_j^{(q-1)}(\lambda)+1-q}
   \right)_{\substack{1 \leq i\leq n \\ 1\leq j\leq n_{q-1}(\lambda)}}   & 
   \left (x_i^{\beta_j^{(t-1-q)}(\lambda)+1+q}
   \right)_{\substack{1 \leq i\leq n \\ 1\leq j\leq n_{t-1-q}(\lambda)}} 
   \end{array}\right) \\
   = &\det  \left(
   \begin{array}{c|c}
      A^{\lambda}_{q-1,1-q}  & \bar{A}^{\lambda}_{t-q-1,q+1} \\[0.2cm]
   \hline \\[-0.3cm]
       \bar{A}^{\lambda}_{q-1,1-q}  &  A^{\lambda}_{t-q-1,q+1}
   \end{array}
   \right).  
    \end{split}
\end{equation}
Combining \eqref{num} and \eqref{spdet5}, {we see that} the numerator of \eqref{symp-1} is given by 
\begin{equation}
    \begin{split}
    \label{num11}
\sgn(\sigma_{\lambda}) &
((-1)^{(t-1)n}t)^{n} 
\det \left(
A^{\lambda}_{t-1,1}-\bar{A}^{\lambda}_{t-1,1}
\right) 
(-1)^{\Sigma_1} 
(\det (\gamma_{i,j})_{1 \leq i,j \leq t-1} )^n \\ &
\times \ds \prod_{q=1}^{\floor{\frac{t-1}{2}}}
\det\left(\begin{array}{c|c}
   A^{\lambda}_{q-1,1-q}  & \bar{A}^{\lambda}_{t-q-1,q+1} \\[0.2cm]
   \hline \\[-0.3cm]
  \bar{A}^{\lambda}_{q-1,1-q}  & A^{\lambda}_{t-q-1,q+1}
\end{array}\right) \times
\begin{cases}
\det \left(
A^{\lambda}_{\frac{t}{2}-1,1}- \bar{A}^{\lambda}_{\frac{t}{2}-1,1} 
\right)  & t \text{ even,}\\
1 & t \text{ odd}.
\end{cases}.     
    \end{split}
\end{equation}
Evaluating \eqref{num11} at the empty partition and using \eqref{sigma0}, {we see that} the denominator of \eqref{symp-1} is given by  
\begin{equation}
    \begin{split}
        \label{spdet4}
(-1)^{\frac{t(t-1)}{2}\frac{n(n+1)}{2}}  ((-1)^{(t-1)n}t)^{n} 
\det \left(A_{t-1,1}-\bar{A}_{t-1,1}\right) & 
(-1)^{\Sigma^0_{1}} (\det (\gamma_{i,j})_{1 \leq i,j \leq t-1} )^n \\\
\times 
\ds \prod_{q=1}^{\floor{\frac{t-1}{2}}}\det\left(\begin{array}{c|c}
   A_{q-1,1-q}  & \bar{A}_{t-q-1,q+1} \\[0.2cm]
   \hline \\[-0.3cm]
  \bar{A}_{q-1,1-q}  & A_{t-q-1,q+1}
\end{array}\right) & \times
\begin{cases}\det \left(A_{\frac{t}{2}-1,1}- \bar{A}_{\frac{t}{2}-1,1} \right)  & t \text{ even,}\\
1 & t \text{ odd,}
\end{cases}  
    \end{split}
\end{equation}
where 
\[
\Sigma^0_{1}
=\sum_{q=1}^{\floor{\frac{t-1}{2}}}
2n
+\begin{cases}
0 & t \text{ odd},\\
n \ds \sum_{q=1}^{\frac{t-2}{2}} n  &  t \text{ even}.
\end{cases}
\]

For $0 \leq i \leq \floor{\frac{t-3}{2}}$, let $\ds \mu^{(1)}_i 
=  \lambda^{(t-2-i)}_1 + \left(\lambda^{(i)}, 0,\dots,0, -\rev(\lambda^{(t-2-i)})\right)$. Since
$n_{i}(\lambda)+n_{t-2-i}(\lambda) =  2n$, \cref{lem:s-new} gives 
\begin{equation}
\label{scfinal}
s_{\mu^{(1)}_i}(X^t,{\X}^t)= \frac{(-1)^{\frac{n_{t-i-2}(\lambda)(n_{t-i-2}(\lambda)-1)}{2}}}{(-1)^{\frac{n(n-1)}{2}}}
   \frac{
  \det
   \left(\begin{array}{c|c}
   A^{\lambda}_{i,-i}  & \bar{A}^{\lambda}_{t-2-i,i+2} \\[0.2cm]
   \hline \\[-0.3cm]
  \bar{A}^{\lambda}_{i,-i}  & A^{\lambda}_{t-2-i,i+2}
\end{array}\right)}
{
\det
\left(\begin{array}{c|c}
   A_{i,-i}  & \bar{A}_{t-2-i,i+2} \\[0.2cm]
   \hline \\[-0.3cm]
  \bar{A}_{i,-i}  & A_{t-2-i,i+2}
\end{array}\right)}.
\end{equation}
Now substitute \eqref{num11} and \eqref{spdet4} in \eqref{symp-1}, and then use \eqref{sp-new} for $p=t-1$, \eqref{oo-new} for $p=\frac{t}{2}-1$ and \eqref{scfinal} for $0 \leq i \leq \floor{\frac{t-3}{2}}$.
The symplectic character is thus given by 
\begin{multline*}
\sp_{\lambda}(X,\omega X,\dots ,\omega^{t-1}X) \\ 
= (-1)^{\epsilon}
\sgn(\sigma_{\lambda})
\sp_{\lambda^{(t-1)}}(X^t)
\prod_{i=0}^{\floor{\frac{t-3}{2}}} 
s_{\mu^{(1)}_i}(X^t,{\X}^t) \times 
\begin{cases}
\oo_{\lambda^{\left(\frac{t}{2}-1\right)}}(X^t) & t \text{ even},\\
1 & t \text{ odd},
\end{cases}
\end{multline*}
where 
\begin{multline*}
\epsilon= 
\frac{t(t-1)}{2}\frac{n(n+1)}{2} 
+ \ds \sum_{q=0}^{\floor{\frac{t-3}{2}}}
\left(n_{q}(\lambda)-n\right) 
\times 
\begin{cases}
n+1 &  t \text{ even}\\
1 &  t \text{ odd}
\end{cases} \\
+  \sum_{i=0}^{\floor{\frac{t-3}{2}}}
\left( 
\frac{n(n-1)}{2} - \frac{n_{t-i-2}(\lambda)(n_{t-i-2}(\lambda)-1)}{2} \right).
\end{multline*}
It remains to compute the sign by simplifying the expression for $\epsilon$.
Since for $0 \leq q \leq \floor{\frac{t-3}{2}}$, $n_q(\lambda)+n_{t-2-q}(\lambda)=2n$, replacing $n_q(\lambda)-n$ by $n-n_{t-2-q}(\lambda)$ in the first summation and then using the facts that $\frac{(t-1)(t+1)}{2}\frac{n(n+1)}{2}$ is even for odd $t$ and {the parity of $\frac{n(n+1)}{2}\frac{(t^2-2)}{2}$ is the same as the parity} of $\frac{n(n+1)}{2}$ for odd $t$ shows that $\epsilon$ has the same parity as
\begin{multline*}
-\sum_{i=0}^{\floor{\frac{t-3}{2}}} \binom{n_{t-2-i}(\lambda)+1}{2} + \begin{cases}
\frac{n(n+1)}{2}+nr & t \text{ even},
\\ 0 & t \text{ odd},
\end{cases}\\
=-\sum_{i=\floor{\frac{t}{2}}}^{t-2} \binom{n_{i}(\lambda)+1}{2} + \begin{cases}
\frac{n(n+1)}{2}+nr & t \text{ even},
\\ 0 & t \text{ odd},
\end{cases}
\end{multline*}
where $r$ is the rank from \cref{lem:rank-sympcore-etc}(1).
This completes the proof. 
\end{proof}

\subsection{Even orthogonal characters}
\label{sec:eorth}

We first recall the matrices 
$A_{p,q}^{\lambda}$ and $\bar{A}_{p,q}^{\lambda}$ from \eqref{def AB}.
If $\ds \sum_{i=1}^{t-1}n_{i}(\lambda) \allowbreak = (t-1)n$, 
then 
consider the $(t-1)n \times (t-1)n$ block matrix
\[
\Pi_2= \left(\omega^{pq}A^{\lambda}_{q}+\bar{\omega}^{pq}
\bar{A}^{\lambda}_{q}\right)_{1 \leq p,q \leq t-1}.
\]
{Substitution of} $U_j=A_j^{\lambda}$, $V_j=-\bar{A}_j^{\lambda}$ and for $1 \leq j \leq t-1$,
\[
\gamma_{i,j} = \begin{cases}
\omega^{\frac{i(j+1)}{2}} & j \text{ odd },\\
\omega^{-\frac{ij}{2}} & j \text{ even }, 
\end{cases}
\]
in \cref{lem:det-blockmatrix} proves the following corollary. 

\begin{cor} 
\label{cor:det-even}
\begin{enumerate}
\item If $n_{i}(\lambda)+n_{t-i}(\lambda)\neq 2n$ 
for some $i \in [\floor{\frac{t}{2}}]$, then
$\det \Pi_2 = 0$.

\item If $n_{i}(\lambda)+n_{t-i}(\lambda) = 2n$ for all $i \in [\floor{\frac{t}{2}}]$, then 
\begin{equation}
\begin{split}
   \label{eodet3}
 \det \Pi_2 
= (-1)^{\Sigma_2}  \left(\det (\gamma_{i,j})_{1 \leq i,j \leq t-1} \right)^{n} &
\prod_{q=1}^{\floor{\frac{t-1}{2}}}
\det\left(\begin{array}{c|c}
   A^{\lambda}_{q}  & \bar{A}^{\lambda}_{t-q} \\[0.2cm]
   \hline \\[-0.3cm]
   \bar{A}^{\lambda}_{q}  & A^{\lambda}_{t-q}
   \end{array}\right) \\ &
\times
\begin{cases}
\det \left(A^{\lambda}_{\frac{t}{2}}+ \bar{A}^{\lambda}_{\frac{t}{2}} \right)  & t \text{ even}\\
1 & t \text{ odd}
\end{cases}    
\end{split}
\end{equation}
\end{enumerate}
where 
\[
\Sigma_2 =
\begin{cases}
n \ds \sum_{q=1}^{\frac{t-2}{2}} n_{q}(\lambda) 
&  t \text{ even},\\
0 &  t \text{ odd}.
\end{cases}
\]
\end{cor}

We now give a sketch of the proof of \cref{thm:eorthfact} following similar ideas as in the proof of \cref{thm:sympfact}.

\begin{proof}[Proof of \cref{thm:eorthfact}]
Using the formula for even orthogonal characters is \eqref{oedef}, {we see that}
desired polynomial is
\begin{equation}
 \label{eors}
\oe_{\lambda}(X,\omega X,\dots ,\omega^{t-1}X) =
\frac{ {2}
\det\left( \left( (\omega^{p-1}x_i)^{\beta_{j}(\lambda)}
    +(\bar{\omega}^{p-1}\bar{x}_i)^{\beta_{j}(\lambda)}
    \right)_{\substack{1 \leq i\leq n \\ 1\leq j\leq tn}} 
    \right)_{1\leq p \leq t}}
    {{(1+\delta_{\lambda_{tn},0})} \det\left( \left( (\omega^{p-1}x_i)^{tn-j}
    +(\bar{\omega}^{p-1}\bar{x}_i)^{tn-j}
    \right)_{\substack{1 \leq i\leq n \\ 1\leq j\leq tn}}
    \right)_{1\leq p \leq t}}.   
\end{equation}
After permuting the columns of the determinant in the numerator of \eqref{eors} by $\sigma_{\lambda}$ from \eqref{sigma-perm}, {we see that}
the numerator of  \eqref{eors} becomes
\begin{equation}
 \label{eofinal}
  {2} \sgn(\sigma_{\lambda}) 
  \det \left(
  \omega^{(p-1)(q-1)} A^{\lambda}_{q-1}+\bar{\omega}^{(p-1)(q-1)}\bar{A}^{\lambda}_{q-1}
  \right)_{1\leq p,q\leq t}.
 \end{equation}
By applying the block operation
$ R_1 \rightarrow R_1+R_2+\dots+R_{t}$ 
and then $R_i \rightarrow R_i-\frac{1}{t}R_1$, $2 \leq i \leq t$, {we see that}
the numerator of  \eqref{eors} is
\begin{equation*} 
{2} \sgn(\sigma_{\lambda})
\det\left(\begin{array}{c|c|c|c|c}
A^{\lambda}_{0}+\bar{A}^{\lambda}_{0} & 0 & \dots & 0 & 0\\
&&&\\\hline&&&\\
0 & \omega A^{\lambda}_{1}+\omega^{t-1}\bar{A}^{\lambda}_{1} & \dots & \omega^{t-2}A^{\lambda}_{t-2}+\omega^2 \bar{A}^{\lambda}_{t-2} & \omega^{t-1}A^{\lambda}_{t-1}+\omega \bar{A}^{\lambda}_{t-1} \\
&&&\\\hline&&&\\
0 & \omega^2A^{\lambda}_{1}+\omega^{t-2}\bar{A}^{\lambda}_{1} & \dots & \omega^{t-4}A^{\lambda}_{t-2}+\omega^4 \bar{A}^{\lambda}_{t-2}&\omega^{t-2}A^{\lambda}_{t-1}+\omega^2\bar{A}^{\lambda}_{t-1}\\&&&\\\hline&&&\\
\vdots & \vdots & \ddots & \vdots & \vdots \\&&&\\\hline&&&\\
0 & \omega^{t-1}A^{\lambda}_{1}+\omega \bar{A}^{\lambda}_{1} & \dots & \omega^2 A^{\lambda}_{t-2}+\omega^{t-2} \bar{A}^{\lambda}_{t-2}&\omega A^{\lambda}_{t-1}+\omega^{t-1}\bar{A}^{\lambda}_{t-1}
\end{array}\right)
\end{equation*} 
This is a $2 \times 2$ block diagonal matrix. 

If $\core \lambda t$ is not an orthogonal $t$-core, then by \cref{cor:ocore},
either $n_{0}(\lambda) \neq n$ or  $n_{i}(\lambda)+n_{t-i}(\lambda) \neq  2n$ for some $i \in [\floor{\frac{t}{2}}]$.
If $n_{0}(\lambda) \neq n$, then the above determinant is 0 by \cref{lem:matrix}. 
If $n_{0}(\lambda) = n$, then the numerator of  \eqref{eors} is
\begin{multline}
 \label{eodet2}
{2} \sgn(\sigma_{\lambda}) t^n \det \left(A^{\lambda}_{0}+\bar{A}^{\lambda}_{0}\right)
\det \left( \left(\omega^{(p-1)(q-1)}A^{\lambda}_{q-1}+\bar{\omega}^{(p-1)(q-1)}\bar{A}^{\lambda}_{q-1}\right)_{2 \leq p,q \leq t}\right).
\end{multline}
where the last determinant in \eqref{eodet2} is the determinant of $\Pi_2$, computed in \cref{cor:det-even}. 
If $n_{i}(\lambda)+n_{t-i}(\lambda) \neq  2n$ for some $i \in [\floor{\frac{t}{2}}]$,
then this is $0$ by \cref{cor:det-even}(1). In both cases,
\[
\oe_{\lambda}(X,\omega X,\dots ,\omega^{t-1}X) = 0.
\]

If $\core \lambda t$ is an orthogonal $t$-core, then
by \cref{cor:ocore}, $n_{0}(\lambda) = n$ and $n_{i}(\lambda)+n_{t-i}(\lambda) =  2n$, $i \in [\floor{\frac{t}{2}}]$.
Using \eqref{eodet2} and \cref{cor:det-even}(2), {we see that} the numerator of  \eqref{eors} is
\begin{equation}
    \begin{split}
     \label{eodet31}
 {2} \sgn(\sigma_{\lambda})
 t^n
 \det \left(A^{\lambda}_{0}+\bar{A}^{\lambda}_{0}\right)
 &\det((\gamma_{i,j})_{1 \leq i,j \leq t-1} )^n
 (-1)^{\Sigma_2}  \\ &\times \prod_{q=1}^{\floor{\frac{t-1}{2}}}
 \det\left(\begin{array}{c|c}
   A^{\lambda}_{q}  & \bar{A}^{\lambda}_{t-q} \\[0.2cm]
   \hline \\[-0.3cm]
   \bar{A}^{\lambda}_{q}  & A^{\lambda}_{t-q}
\end{array}\right) 
\times
\begin{cases}
\det \left(A^{\lambda}_{\frac{t}{2}}+ \bar{A}^{\lambda}_{\frac{t}{2}} \right)  & t \text{ even,}\\
1 & t \text{ odd}.
\end{cases}    
    \end{split}
\end{equation}

 The rest of the proof proceeds in almost complete analogy with the proof of \cref{thm:sympfact}.
Using \eqref{sigma0} {and the fact that $\lambda_{tn} = 0$ if and only if $\lambda^{(0)}_{n} = 0$}, {we see that} the denominator of \eqref{eors} is 
\begin{equation}
    \begin{split}
       \label{eodet4}
(-1)^{\frac{t(t-1)}{2}\frac{n(n+1)}{2}}
t^n & {(1+\delta_{\lambda^{(0)}_{n},0})} \det \left(A_{0}+\bar{A}_{0}\right)
 \det((\gamma_{i,j})_{1 \leq i,j \leq t-1} )^n (-1)^{\Sigma^0_{2}} \\
\times &
 \prod_{q=1}^{\floor{\frac{t-1}{2}}}
 \det\left(\begin{array}{c|c}
   A_{q,-q}  & \bar{A}_{t-q,q} \\[0.2cm]
   \hline \\[-0.3cm]
   \bar{A}_{q,-q}  & A_{t-q,q}
\end{array}\right) \times
\begin{cases}
\det \left(A_{\frac{t}{2}}+ \bar{A}_{\frac{t}{2}} \right)  & t \text{ even}\\
1 & t \text{ odd}
\end{cases}   
    \end{split}
\end{equation}
where 
\[
\Sigma^0_{2}=
\begin{cases}
n \ds \sum_{q=1}^{\frac{t-2}{2}} n &  t \text{ even}.\\
0 &  t \text{ odd}.
\end{cases}
\] 
Taking ratios, {we see that} {one of the factors is exactly} the even orthogonal character of $\lambda^{(0)}$ and the $i$'th determinant in the product of \eqref{eodet31} is calculated using
\begin{equation}
\label{efinal}
s_{\mu^{(2)}_i}(X^t,{\X}^t)= \frac{(-1)^{\frac{n_{t-i}(\lambda)(n_{t-i}(\lambda)-1)}{2}}}{(-1)^{\frac{n(n-1)}{2}}}
   \frac{
  \det
   \left(\begin{array}{c|c}
   A^{\lambda}_{i,-i}  & \bar{A}^{\lambda}_{t-i,i} \\[0.2cm]
   \hline \\[-0.3cm]
  \bar{A}^{\lambda}_{i,-i}  & A^{\lambda}_{t-i,i}
\end{array}\right)}
{
\det
\left(\begin{array}{c|c}
   A_{i,-i}  & \bar{A}_{t-i,i} \\[0.2cm]
   \hline \\[-0.3cm]
  \bar{A}_{i,-i}  & A_{t-i,i}
\end{array}\right)},   
\end{equation}
where $\ds \mu^{(2)}_i
=  \lambda^{(t-i)}_1 + \left(\lambda^{(i)}, 0,\dots,0, -\rev(\lambda^{(t-i)})\right)$.
The only new part is the final determinant, which is calculated
using \eqref{eo-new} and \eqref{oeshifted}, and we get 
\begin{equation}
\label{eo12}  
\frac
{\det \left(A^{\lambda}_{\frac{t}{2}}+ \bar{A}^{\lambda}_{\frac{t}{2}} \right)}
{\det \left(A_{\frac{t}{2}}+ \bar{A}_{\frac{t}{2}} \right)}
=\frac{\oe_{\lambda^{(t/2)}+1/2}(X^t)}
{\ds \prod_{i=1}^n (x_i^{t/2}+\bar{x}_i^{t/2})} = (-1)^{\sum_{i=1}^n \lambda_i^{(t/2)}} \ds {\oo_{\lambda^{(t/2)}}(-X^t)}.
\end{equation}
Finally, the even orthogonal character is given by
\begin{multline*}
\oe_{\lambda}(X,\omega X, \dots ,\omega^{t-1}X) \\ =  (-1)^{\epsilon}
\sgn(\sigma_{\lambda})
\oe_{\lambda^{(0)}}(X^t)
\prod_{i=1}^{\floor{\frac{t-1}{2}}}
s_{\mu^{(2)}_i}(X^t,{\X}^t) 
\times
  \begin{cases}
(-1)^{ \sum_{i=1}^n \lambda_i^{(t/2)}} \displaystyle {\oo_{\lambda^{(t/2)}}(-X^t)} & t \text{ even},\\
1 & t \text{ odd},
\end{cases}
\end{multline*}
where 
\begin{multline*}
\epsilon= 
\frac{t(t-1)}{2}\frac{n(n+1)}{2} 
+ \ds \sum_{q=1}^{\floor{\frac{t-1}{2}}}
\left(n_{q}(\lambda)-n\right) 
\times 
\begin{cases}
n &  t \text{ even}\\
0 &  t \text{ odd}
\end{cases} \\
+  \sum_{i=1}^{\floor{\frac{t-1}{2}}}
\left( 
\frac{n(n-1)}{2} - \frac{n_{t-i}(\lambda)(n_{t-i}(\lambda)-1)}{2} \right).
\end{multline*}
After similar simplifications, the parity of $\epsilon$ shown to be the same as
\[
-\sum_{i=\floor{\frac{t+2}{2}}}^{t-1} \binom{n_{i}(\lambda)}{2} + \begin{cases}
\frac{n(n+t-1)}{2}+nr & t \text{ even},\\ 
\frac{(t-1)n}{2} & t \text{ odd},
\end{cases} 
\]
where $r$ is the rank by \cref{lem:rank-sympcore-etc}(2), completing the proof.
\end{proof}

\subsection{Odd orthogonal characters}
\label{sec:oorth}

Recall the matrices 
$A_{p,q}^{\lambda}$ and $\bar{A}_{p,q}^{\lambda}$ from \eqref{def AB}.
Consider the $tn \times tn$ block matrix
\[ 
\Pi_3=
\left( 
\omega^{(p-1)q} A^{\lambda}_{q-1,1}
-\bar{\omega}^{(p-1)(q-1)}\bar{A}^{\lambda}_{q-1,0}
\right)_{1\leq p,q\leq t}.
\]
{Substitution of} $U_j=A_{j-1,1}^{\lambda}$, $V_j=\bar{A}_{j-1,0}^{\lambda}$ for $1 \leq j \leq t$ and
\[
\gamma_{i,j} = \begin{cases}
\omega^{\frac{(i-1)(j+1)}{2}} & j \text{ odd }\\
\omega^{-\frac{(i-1)(j-2)}{2}} & j \text{ even } 
\end{cases}
\]
in \cref{lem:det-blockmatrix} proves the following corollary.

\begin{cor} 
\label{cor:det-odd}
\begin{enumerate}
\item If $n_{i}(\lambda)+n_{t-1-i}(\lambda)\neq 2n$ for some $i \in [0,\floor{\frac{t-1}{2}}]$, then $\det \Pi_3 = 0$.

\item If $n_{i}(\lambda)+n_{t-1-i}(\lambda) = 2n$ for all $i \in \{0,1,\dots,\floor{\frac{t-1}{2}}\}$, then 
\begin{equation}
    \begin{split}
        \label{odet3}
 \det \Pi_3 
= (\det (\gamma_{i,j})_{1 \leq i,j \leq t} )^{n} (-1)^{\Sigma_{3}} \prod_{q=1}^{\floor{\frac{t}{2}}}
& \det\left(\begin{array}{c|c}
   A^{\lambda}_{q-1,1}  & \bar{A}^{\lambda}_{t-q,0} \\[0.2cm]
   \hline \\[-0.3cm]
   \bar{A}^{\lambda}_{q-1,0}  & A^{\lambda}_{t-q,1}
\end{array}\right)  \\
\times &
\begin{cases}
\det \left(A^{\lambda}_{\frac{t-1}{2},1}- \bar{A}^{\lambda}_{\frac{t-1}{2},0} \right)  & t \text{ odd},\\
1 & t \text{ even},
\end{cases}   
    \end{split}
\end{equation}
where 
\[
\Sigma_{3}
=\sum_{q=1}^{\floor{\frac{t}{2}}} 
\left(n+n_{q-1}(\lambda)\right)+\begin{cases}
n \ds \sum_{q=1}^{\frac{t-1}{2}} n_{q-1}(\lambda) &  t \text{ odd},\\
0 &  t \text{ even.}
\end{cases}
\]
\end{enumerate}
\end{cor}

\begin{proof}[Proof of \cref{thm:oorthfact}]
Starting from the formula for the odd orthogonal character in \eqref{oodef}, {we see that} the desired polynomial is
\begin{equation}
\label{so}
 \oo_{\lambda}(X,\omega X, \dots ,\omega^{t-1}X) 
 = \frac{\det\left( \left( (\omega^{p-1}x_i)^{\beta_{j}(\lambda)+1}
 -(\bar{\omega}^{p-1}\bar{x}_i)^{\beta_{j}(\lambda)}
 \right)_{\substack{1 \leq i\leq n \\ 1\leq j\leq tn}} 
 \right)_{1\leq p\leq t}}
 {\det \left( \left( (\omega^{p-1}x_i)^{tn-j+1}-(\bar{\omega}^{p-1}\bar{x}_i)^{tn-j} \right)_{\substack{1 \leq i\leq n \\ 1\leq j\leq tn}} \right)_{1\leq p\leq t}}.
\end{equation}
 We again proceed as in the proof of \cref{thm:sympfact}. 
 Permuting the columns of the determinant in the numerator in \eqref{so} by the permutation $\sigma_{\lambda}$ from \eqref{sigma-perm}, {we see that} the numerator in \eqref{so} is 
\begin{equation}
\label{oofinal}
\sgn(\sigma_{\lambda})\det \left( \omega^{(p-1)q} A^{\lambda}_{q-1,1}-\bar{\omega}^{(p-1)(q-1)}\bar{A}^{\lambda}_{q-1}\right)_{1\leq p,q\leq t}, \end{equation}
where the last determinant in \eqref{eodet2} is the determinant of $\Pi_2$, computed in \cref{cor:det-odd}.
If $\core \lambda t$ is not self-conjugate, then by \cref{cor:con},
$n_{i}(\lambda)+n_{t-1-i}(\lambda) \neq  2n$ for some $i \in \{0,1,\dots,\floor{\frac{t-1}{2}}\}$.
In that case, the determinant is 0 by \cref{cor:det-odd}(1). Hence
\[
\oo_{\lambda}(X,\omega X,\dots ,\omega^{t-1}X) = 0.
\]

If $\core \lambda t$  is self-conjugate, then by \cref{cor:con}, $n_{i}(\lambda)+n_{t-1-i}(\lambda) =  2n$ for all $i \in \{0,1,\dots,\floor{\frac{t-1}{2}}\}$.
By \cref{cor:det-odd}(2), the numerator in \eqref{so} is
\begin{equation}
     \label{oodet3}
\sgn(\sigma_{\lambda}) (\det (\gamma_{i,j})_{1 \leq i,j \leq t})^n (-1)^{\Sigma_{3}} \prod_{q=1}^{\floor{\frac{t}{2}}}
\det\left(\begin{array}{c|c}
   A^{\lambda}_{q-1,1}  & \bar{A}^{\lambda}_{t-q} \\[0.2cm]
   \hline \\[-0.3cm]
   \bar{A}^{\lambda}_{q-1}  & A^{\lambda}_{t-q,1}
\end{array}\right) \times
\begin{cases}
1 & \hspace*{-0.15cm} t \text{ even},\\
\det \left(A^{\lambda}_{\frac{t-1}{2},1}- \bar{A}^{\lambda}_{\frac{t-1}{2}} \right)  & \hspace*{-0.15cm} t \text{ odd}.
\end{cases}  
\end{equation}
We now evaluate the $2 \times 2$ block determinant as follows: for $1 \leq q \leq \floor{\frac{t}{2}}$, we multiply row $i$ in the top blocks by $\bar{x}_i^{q}$ and row $i$ in the bottom blocks by $x_i^{q-1}$, for each $i$.
We then end up with
\begin{equation}
    \begin{split}
        \label{o31}
\sgn(\sigma_{\lambda})
     (\det (\gamma_{i,j})_{1 \leq i,j \leq t})^n
     (-1)^{\Sigma_{3}}
     \prod_{q=1}^{\floor{\frac{t}{2}}}
     \left(\x_1 \x_2 \dots \x_n \det \left(\begin{array}{c|c}
   A_{q-1,1-q}^{\lambda}  &  \bar{A}_{t-q,q}^{\lambda}  \\[0.2cm]
   \hline \\[-0.3cm]
 \bar{A}_{q-1,1-q}^{\lambda}    &  A_{t-q,q}^{\lambda}  \end{array}\right)\right) \\
 \times
\begin{cases}
1 & t \text{ even},\\
\det \left(A^{\lambda}_{\frac{t-1}{2},1}- \bar{A}^{\lambda}_{\frac{t-1}{2}} \right)  & t \text{ odd}.
\end{cases}
    \end{split}
\end{equation}
The denominator in \eqref{so} is therefore
\begin{equation}
    \begin{split}
       \label{oodet8}
 (-1)^{\frac{t(t-1)}{2}\frac{n(n+1)}{2}}
 (\det (\gamma_{i,j})_{1 \leq i,j \leq t})^n  (-1)^{\Sigma_{3}^0} \prod_{q=1}^{\floor{\frac{t}{2}}} \left(\x_1\x_2\dots\x_n \det \left(\begin{array}{c|c}
   A_{q-1,1-q} &  \bar{A}_{t-q,q} \\[0.2cm]
   \hline \\[-0.3cm]
 \bar{A}_{q-1,1-q}   &  A_{t-q,q} \end{array}\right) \right)\\
 \times
\begin{cases}
1 & t \text{ even},\\
\det \left(A_{\frac{t-1}{2},1}- \bar{A}_{\frac{t-1}{2}}\right)  & t \text{ odd},
\end{cases}     
    \end{split}
\end{equation}
where 
\[
\Sigma^0_{3}=\sum_{q=1}^{\floor{\frac{t}{2}}} 2n
+\begin{cases}
n \ds \sum_{q=1}^{\frac{t-1}{2}} n&  t \text{ odd},\\
0 &  t \text{ even}.
\end{cases}
\]
Taking ratios, {we see that} the block determinants are proportional to Schur functions 
using \cref{lem:s-new},
\begin{equation}
\label{ofinal}
s_{\mu^{(3)}_i}(X^t,{\X}^t)= \frac{(-1)^{\frac{n_{t-1-i}(\lambda)(n_{t-1-i}(\lambda)-1)}{2}}}{(-1)^{\frac{n(n-1)}{2}}}
   \frac{
  \det
   \left(\begin{array}{c|c}
   A^{\lambda}_{i,-i}  & \bar{A}^{\lambda}_{t-1-i,i+1} \\[0.2cm]
   \hline \\[-0.3cm]
  \bar{A}^{\lambda}_{i,-i}  & A^{\lambda}_{t-1-i,i+1}
\end{array}\right)}
{
\det
\left(\begin{array}{c|c}
   A_{i,-i}  & \bar{A}_{t-i,i} \\[0.2cm]
   \hline \\[-0.3cm]
  \bar{A}_{i,-i}  & A_{t-i,i}
\end{array}\right)},   
\end{equation}
where $\ds \mu^{(3)}_i 
=  \lambda^{(t-1-i)}_1 + \left(\lambda^{(i)}, 0,\dots,0, -\rev(\lambda^{(t-1-i)})\right)$. The last ratio of determinants gives an odd orthogonal character.
Finally, the odd orthogonal character is given by
\begin{equation*}
 \oo_{\lambda}(X,\omega X, \dots ,\omega^{t-1}X) =  (-1)^{\epsilon}\sgn(\sigma_{\lambda}) 
  \prod_{i=0}^{\floor{\frac{t-2}{2}}}  
  s_{\mu^{(3)}_i}(X^t,{\X}^t)  \times
\begin{cases}
\oo_{\lambda^{\left(\frac{t-1}{2}\right)}}(X^t) & t \text{ odd},\\
1 & t \text{ even},
\end{cases}   
\end{equation*}
where 
\begin{multline*}
\epsilon= 
\frac{t(t-1)}{2}\frac{n(n+1)}{2} 
+ \ds \sum_{q=0}^{\floor{\frac{t-2}{2}}}
\left(n_{q}(\lambda)-n\right) 
\times 
\begin{cases}
n+1 &  t \text{ odd},\\
1 &  t \text{ even},
\end{cases} \\
+  \sum_{i=0}^{\floor{\frac{t-2}{2}}}
\left( 
\frac{n(n-1)}{2} - \frac{n_{t-1-i}(\lambda)(n_{t-1-i}(\lambda)-1)}{2} \right).
\end{multline*}
After similar simplifications, $\epsilon$ turns out to have the same parity as
\[
-\sum_{i=\floor{\frac{t}{2}}}^{t-1} \binom{n_{i}(\lambda)+1}{2} + \begin{cases}
nr & t \text{ odd}
\\ 0 & t \text{ even}
\end{cases} 
\]
where $r$ is the rank by \cref{lem:rank-sympcore-etc}(3), completing the proof.
\end{proof}

\section{Generating functions}
\label{sec:gf}

We now give enumerative results for $z$-asymmetric partitions defined in \cref{def:z-asym}. We first recall that the \emph{$q$-Pochhammer symbol} is given by
\begin{equation}
(a; q)_m = \prod_{j=0}^{m-1} (1 - a q^j),
\end{equation}
so that $(a; q)_0 = 1$. We also define the limiting infinite product
\begin{equation}
(a; q)_\infty = \prod_{j=0}^{\infty} (1 - a q^j).
\end{equation}
Many generating functions in the theory of partitions can be naturally expressed in terms of $q$-Pochhammer symbols. For example, the generating function for all partitions is $1/(q; q)_\infty$ and that of strict partitions is $(-q; q)_\infty$.

\begin{prop}
\label{prop:bij}
The number of $z$-asymmetric partitions of $m$ 
is equal to the number of partitions of $m$ 
with distinct parts of the form $2k+1+z$, $k \geq 0$.
\end{prop}

\begin{proof}
To prove the proposition, we construct a bijection from the set $\mathcal{P}_z$ 
to the set of partitions of $n$ with distinct parts of the form $2k+1+z$, $k \geq 0$. 
If $\lambda=(\alpha|\alpha+z)$ is a $z$-symmetric partition of rank $r$,
then define $\mu$ of length $r$ by $\mu_i = 2\alpha_i+z+1$. Then all the parts of $\mu$ are distinct and of the desired form. This map is clearly invertible.
\end{proof}

\noindent
\cref{prop:bij} immediately gives an expression of the generating function for $z$-asymmetric partitions.

\begin{cor}
\label{cor:gf}
For $z \in \mathbb{Z}$,
\[
\sum_{\lambda \in \mathcal{P}_z}
q^{|\lambda|} = \prod_{k \geq 0} (1 + q^{z + 1 + 2k})=(-q^{z+1};q^2)_{\infty}.
\]
\end{cor}

We now move on to enumerating $z$-asymmetric partitions which are also $t$-cores. Recall from \cref{lem:z-large} that there are no nontrivial partitions if $z > t - 2$.

\begin{thm}
\label{thm:bijection}
Let $z \leq t - 2$. 
Represent elements of $\mathbb{Z}^{\floor{\frac{t-z}{2}}}$ by $(z_0,\dots, z_{\floor{\frac{t-z-2}{2}}})$ and define $b \in \mathbb{Z}^{\floor{\frac{t-z}{2}}}$ by $\Vec{b}_i=t-z-1-2i$.
Then there exists a bijection 
$\phi : \mathcal{P}_{z,t} \rightarrow \mathbb{Z}^{\floor{\frac{t-z}{2}}}$
satisfying $|\lambda|=t ||\Vec{\phi(\lambda)}||^2-\Vec{b} \cdot \Vec{\phi(\lambda)}$, where $\cdot$ represents the standard inner product.
\end{thm}

\begin{proof}
Suppose $\lambda \in \mathcal{P}_{z,t}$, of length at most $tn$ for some $n\geq 1$. 
Define the map $\phi$ by
\[
(\phi(\lambda))_i  
\coloneqq
n_{i}(\lambda)-n, 
\quad 0 \leq i \leq \floor{\frac{t-z-2}{2}}.
\]
Since $n$ is not unique, it is not {a priori} clear that $\phi$ is well-defined.
But from the definition of {$n_{i}(\lambda)$,} it is easy to see that
$n_{i}(\lambda)-n=n_{i}(\lambda,tn+t)-n-1$. 
Hence, $\phi(\lambda)$ is indeed well-defined.

To show that $\phi$ is a bijection, we define the inverse of $\phi$ as follows. 
For a vector $\Vec{v} = \left(v_0,v_1,\dots,v_{\floor{\frac{t-z-2}{2}}}\right)$,
let $n=\text{max}\{|v_0|,|v_1|,\dots,|v_{\floor{\frac{t-z-2}{2}}}|\}$ 
and for $0 \leq i \leq t-1$, 
\[
m_i=
\begin{cases}
n+v_i &  0 \leq i \leq \floor{\frac{t-z-2}{2}},\\
n-v_{t-z-1-i}  & \floor{\frac{t-z+1}{2}} \leq i \leq t-z-1,\\
n & \text{ otherwise}
\end{cases}.
\]
By construction, $\ds \sum_{i=0}^{t-1} m_i=tn$, $m_{i}+m_{t-z-1-i}=2n$ for $0 \leq i \leq \floor{\frac{t-z-1}{2}}$, $m_{i}=n$ for $t-z \leq i \leq t-1$.
By \cref{lem:converse: sym}, there is a unique $t$-core $\lambda \in \mathcal{P}_{z,t}$ satisfying $n_{i}(\lambda)=m_i$. and we set
$\phi^{-1}(\Vec{v}) = \lambda$.
Moreover the size of $\lambda$ is computed as
\begin{equation}
\label{lbda}
|\lambda|
= \sum_{i=1}^{tn} \beta_i(\lambda)-\frac{tn(tn-1)}{2}.    
\end{equation}
Since $\lambda$ is a $t$-core, $tj+i$, $0 \leq j \leq n_{i}(\lambda)-1$, $0 \leq i \leq t-1$ are the parts of $\beta(\lambda)$ (see \cref{prop:mcd-t-core-quo}). So, 
\begin{multline*}
 \sum_{i=1}^{tn} \beta_i(\lambda)=\sum_{i=0}^{t-1}
 \left(
 i(n_{i}(\lambda))
 +\frac{n_{i}(\lambda)(n_{i}(\lambda)-1)t}{2}
 \right)\\
= \sum_{i=0}^{t-1}
 \left(
 i(n_{i}(\lambda)-n)
 \right)
 +\frac{tn(t-1)}{2}
 + \frac{t}{2} \sum_{i=0}^{t-1}
 n_{i}(\lambda)^2-\frac{t^2n}{2}.
\end{multline*}
{Substituting this in \eqref{lbda}, we get}
\begin{multline*}
|\lambda|= \sum_{i=0}^{t-1}
\left(i(n_{i}(\lambda)-n)\right)
+\frac{t}{2} \left(\sum_{i=0}^{t-1} n_{i}(\lambda)^2-tn^2\right) \\
= \sum_{i=0}^{t-1}
\left(i(n_{i}(\lambda)-n)\right)
+\frac{t}{2}\sum_{i=0}^{t-1}  \left( n_{i}(\lambda)-n\right)^2.
\end{multline*}
Now observe that
\[
-\Vec{b} \cdot \Vec{v}=\sum_{i=0}^{\floor{\frac{t-z-2}{2}}} (z+1-t+2i)v_i=\sum_{i=0}^{\floor{\frac{t-z-2}{2}}} (z+1-t+2i)(n_{i}(\lambda)-n).
\]
Since $\lambda \in \mathcal{P}_{z,t}$, 
using \cref{lem:converse: sym},  we have
\begin{align*}
-\Vec{b}.\Vec{v} =& \sum_{i=0}^{t-1}i\left(n_{i}(\lambda)-n\right), \\
\sum_{i=0}^{t-1}\left(n_{i}(\lambda)-n \right)^2
= &
2 \sum_{i=0}^{\floor{\frac{t-z-2}{2}}}(n_{i}(\lambda)-n)^2
=
2||\Vec{v}||^2.
\end{align*}
Hence $|\lambda|=t ||\Vec{v}||^2-\Vec{b}.\Vec{v}$, completing the proof.
\end{proof}

Define the \emph{Ramanujan theta function}~\cite[Equation (18.1)]{berndt-1991},
\begin{equation}
f(a,b)= \sum_{n=-\infty}^{\infty} a^{\frac{n(n+1)}{2}}b^{\frac{n(n-1)}{2}},
\end{equation}
which is related to the Jacobi theta function.
There are several nice identities satisfied by $f$. For example, 
$f(a,b) = f(b,a), f(1,a) = 2f(a, a^3)$ and $f(-1,a) = 0$~\cite[Chapter 16, Entry 18]{berndt-1991}. In addition, because of the Jacobi triple product identity, we have~\cite[Chapter 16, Entry 19]{berndt-1991},
\[
f(a,b) = (-a ; a b)_\infty (-b ; a b)_\infty (a b ; a b)_\infty.
\]

Let $p_{z,t}(m)$ be the cardinality of partitions in $\mathcal{P}_{z,t}$ of size $m$.

\begin{cor}
\label{cor:zcore-gf}
For $z \leq t - 2$, we have
\[
\sum_{m \geq 0} p_{z,t}(m) q^m = 
\prod_{i=0}^{\floor{(t-z-2)/2}} f(q^{2i + z+1}, q^{2t-2i-z-1}).
\]
\end{cor}

\begin{proof}
As a consequence of \cref{thm:bijection},
\[
\sum_{m \geq 0} p_{z,t}(m) q^m
= \sum_{\Vec{v} \in \mathbb{Z}^{\floor{\frac{t-z}{2}}}}
\prod_{i=0}^{\floor{\frac{t-z-2}{2}}} 
q^{t v_i^2 -(t-z-1-2i)v_i }
\]
Rewriting the exponent and interchanging the order of summation, {we see that} the generating function becomes
\[
\prod_{i=0}^{\floor{\frac{t-z-2}{2}}} \sum_{v_i \in \mathbb{Z}} 
q^{(2i + z+1)\frac{v_i(v_i+1)}{2}
+(2t-2i-z-1)\frac{v_i(v_i-1)}{2} }=
  \prod_{i=0}^{\floor{(t-z-2)/2}}
f(q^{2i + z+1}, q^{2t-2i-z-1}),
\]
completing the proof.
\end{proof}

We remark that the special case of $z = 0$ (i.e. self-conjugate $t$-cores) in \cref{cor:zcore-gf} was obtained by Garvan--Kim--Stanton~\cite[Equations (7.1a) and (7.1b)]{garvan-kim-stanton-1990}. Thus, our result can be viewed as a generalization of theirs for symplectic and orthogonal partitions, leading to an immediate proof of \cref{thm:inf-cores}.

\section*{Acknowledgements}
{We thank the anonymous referees for their useful comments.}
We thank D. Prasad and A. Prasad for very helpful discussions, and P. Alexandersson for suggesting references.
We acknowledge support from the UGC Centre for Advanced Studies.
AA was partially supported by Department of Science and Technology grant CRG/2021/001592.

\bibliographystyle{alpha} 
\bibliography{Bibliography}

\end{document}